\documentclass[11pt,reqno,draft]{amsart}

\usepackage{amssymb} 
\usepackage{mathrsfs} 
\usepackage{stmaryrd} 
\usepackage{color}

\DeclareMathAlphabet{\mathpzc}{OT1}{pzc}{m}{it} 

\newtheorem{Thm}{Theorem}[section]
\newtheorem{Cor}[Thm]{Corollary}
\newtheorem{Lem}[Thm]{Lemma}
\newtheorem{Prop}[Thm]{Proposition}
\newtheorem{Def}[Thm]{Definition}

\theoremstyle{definition}
\newtheorem{Rem}[Thm]{Remark}

\theoremstyle{definition}

\theoremstyle{definition}

\theoremstyle{definition} 

\makeatletter
\@addtoreset{equation}{section}
\makeatother

\newcommand\Sres{\mathcal{S}}

\newcommand\s{\mathpzc{S}} 
\newcommand{\quat}{\mathbb{H}}
\newcommand{\A}{\mathsf{A}}  
\DeclareMathOperator{\re}{Re} 
\newcommand{\Id}{\ \!\mathsf{Id}}
\newcommand{\Sq}{\mathsf{S}_q^{{}^{-1}}}
\newcommand{\Q}{\mathsf{Q}}
\newcommand{\R}{\mathsf{R}}  

\newcommand\spow{\mathfrak{p}}
\newcommand{\qspow}{\spow_{q_0,n}}

\newcommand\qbar{{\bar q}} 
\newcommand\pbar{{\bar p}} 

\newcommand{\eps}{\varepsilon}

\newcommand\funzione{\longrightarrow}
\newcommand\ci{\mathbb{C}} 
\newcommand\erre{\mathbb{R}} 

\newcommand{\su}{\mathbb{S}}  

\renewcommand\j{\mathbf{j}}  
\renewcommand\k{\mathbf{k}} 

\DeclareMathOperator{\im}{Im} 

\newcommand\setmeno{\!\smallsetminus\!} 

\providecommand{\opint}[1]{\hspace{0.15ex}\left]#1\right[\hspace{0.15ex}} 

\newcommand\norma[2]{\Vert #1\Vert_{#2}} 

\newcommand{\B}{\mathsf{B}}  

\renewcommand\r{\textsl{r}} 

\newcommand\Lin{\mathscr{L}}

\newcommand{\C}{\mathsf{B}}

\newcommand\zbar{\overline{z}}

\newcommand{\F}{\mathsf{F}}  
\newcommand{\G}{\mathsf{G}}  

\newcommand{\M}{\mathsf{M}}  
\newcommand{\N}{\mathsf{N}}  

\providecommand{\opint}[1]{\hspace{0.15ex}\left]#1\right[\hspace{0.15ex}} 
\providecommand{\clsxint}[1]{\hspace{0.1ex}\left[#1\right[\hspace{0.15ex}} 
\DeclareMathOperator{\de}{d \! \hspace{0.2ex}} 

\renewcommand{\S}{\mathsf{S}}  
\newcommand{\Sp}{\mathsf{S}_p^{{}^{-1}}}  

\renewcommand\sp{\hspace{2.8ex}} 

\newcommand\enne{\mathbb{N}} 

\def\vuoto{\varnothing} 

\newcommand{\lra}{\longrightarrow} 

\newcommand{\OO}{\Omega}
\newcommand{\mr}{\mathrm}

\definecolor{blu1}{rgb}{0.1,0.1,1}

\definecolor{verde}{rgb}{0.1,0.4,0.2}

\definecolor{pink}{rgb}{1.0, 0.33, 0.64}

\begin{document}


\title[Quaternionic resolvent equation]{Quaternionic resolvent equation \\ and series expansion of \\ the 
$\Sres$-resolvent operator}

\author{Riccardo Ghiloni and Vincenzo Recupero}

\address{\textbf{Riccardo Ghiloni}\\
        Dipartimento di Matematica\\ 
        Universit\`a di Trento\\
        Via Sommarive 14\\ 
        38123 Trento\\ 
        Italy. \newline
        {\rm E-mail address:}
        {\tt ghiloni@science.unitn.it}}
\address{\textbf{Vincenzo Recupero}\\
        Dipartimento di Scienze Matematiche\\ 
        Politecnico di Torino\\
        C.so Duca degli Abruzzi 24\\ 
        10129 Torino\\ 
        Italy. \newline
        {\rm E-mail address:}
        {\tt vincenzo.recupero@polito.it}}
\thanks{The first author is partially supported by GNSAGA-INdAM. The second author is partially supported by GNAMPA-INdAM}

   
\date{}


\begin{abstract}
In the present paper, we prove a resolvent equation for the $\Sres$-resolvent ope\-rator in the quaternionic framework. Exploiting this resolvent equation, we find a series expansion for the $\Sres$-resolvent operator in an open neighborhood of any given quaternion belonging to the $\Sres$-resolvent set. Some consequences of the series expansion are deduced. In particular, we describe a property of the geometry of the $\Sres$-resolvent set in terms of the Cassini pseudo-metric on quaternions. The concept of vector-valued real analytic function of several variables plays a crucial role in the proof of the mentioned series expansion for the $\Sres$-resolvent operator.
\end{abstract}


\maketitle


\thispagestyle{empty}


\section{Introduction}
The algebra of quaternions $\quat$ was introduced by W.R.\ Hamilton in the midst of the nineteenth century in order to derive elegant and efficient vector analysis formulas for classical mechanics (cf.\ \cite{Ham67}). Nowadays quaternions are still used in classical mechanics, and they also proved to be extremely useful in computer graphics and robotics (see, e.g., \cite{GolPooSaf01, Hog92}).
Another field of application of quaternions is quantum mechanics, indeed G.\ Birkhoff and J.\ von Neumann in their paper \cite{BirNeu36} pointed out that quantum mechanics can be formulated also on Hilbert spaces whose set of scalars is $\quat$, and this possibility was proved to be actually true in \cite{Sol95}. This remark originated the study of quantum mechanics in the quaternionic framework (see, e.g., \cite{FinJauSchSpe62, Emc63, HorBie84, CasTru85, Adl95}), but the full rigorous development of the quaternionic formulation of quantum mechanics was prevented by the lack of a suitable quaternionic notion of spectrum (cf., e.g., the discussions in \cite{CoGeSaSt,GhiMorPer13}).

The first step toward a rigorous formulation of quaternionic quantum mechanics has its basis in the theory of functions of a quaternionic variable developed starting from \cite{GeSt}. In that paper, Gentili and Struppa introduced the notion of slice regular function, which turns out to be an efficient counterpart of the complex concept of holomorphic function (this theory was then developed in several papers: see \cite{GenSto22} and the references therein). Then F.\ Colombo and I.\ Sabadini realized that slice regularity could be conveniently exploited to find the proper notion of quaternionic spectrum in order to establish a rigorous formulation of quaternionic quantum mechanics (see \cite{CoGeSaSt07}). To be more precise, given a right linear operator $\A$ on a quaternionic Banach space $X$, its $\Sres$-resolvent set was defined in \cite{CoGeSaSt07} as the set $\rho_\s(\A)$ of quaternions $q$ such that the inverse of 
\[
  \Delta_q(\A) := \A^2 -2\re(q)\A + |q|^2 \Id
\] 
exists and is bounded. Here $\Id$ denotes the identity on $X$. So they introduced the notion of (left) $\Sres$-resolvent operator 
\begin{align}\label{left sph res op intro}
  \Sq(\A)
  & := \Q_q(\A) \qbar - \A\Q_q(\A),
\end{align}
where 
\[
  \Q_q(\A) := (\A^2 -2\re(q)\A + |q|^2\Id)^{-1}
\]
and $\Sq(\A)$ proved to be the proper kernel which allowed to develop a rigorous quaternionic functional calculus (cf. the papers \cite{CoGeSaSt,CoGeSaSt10,ColSab11,GhiRec16}, the monographs \cite{CoSaSt, ColGanKim18} and the references therein).

In the classical complex theory, the study the resolvent operator 
$\R_\lambda(\A) = (\A - \lambda\Id)^{-1}$ is based on two powerful tools: the resolvent equation 
\begin{equation}\label{res eq class}
\R_\lambda(\A) - \R_\mu(\A) = (\mu-\lambda)\R_\lambda(\A)\R_\mu(\A)
\end{equation}
and  the resolvent power expansion
\begin{equation}\label{C-exp of R}
  \R_\lambda(\A) = \sum_{n=0}^\infty (\mu-\lambda)^n \R_\mu(\A)^{n+1}
\end{equation}
in a neighborhood of $\mu$. These two concepts are still missing in quaternionic functional analysis. A first step to fill this gap was made in \cite{AlpColGanSab15} with the proof of the identity
\begin{align}\label{mix res eq}
 \S_{R,q}^{{}^{-1}}(\A)\S_{L,p}^{{}^{-1}}(\A) = 
 [(\S_{R,q}^{{}^{-1}}(\A)-\S_{L,p}^{{}^{-1}}(\A))p - 
 \qbar(\S_{R,q}^{{}^{-1}}(\A)-\S_{L,p}^{{}^{-1}}(\A))]
 (p^2 - 2\re(q)p+|q|^2)^{-1},
\end{align}
where $\S_{L,p}^{{}^{-1}}(\A) = \Sp(\A) $ is the left $\Sres$-resolvent operator defined in \eqref{left sph res op intro} and 
\begin{equation}\label{righ sph res op intro}
  \S_{R,q}^{{}^{-1}}(\A) := (\qbar\Id-\A)\Q_q(\A) = (\qbar\Id-\A)(\A^2 -2\re(q)\A + |q|^2\Id)^{-1} 
\end{equation}
is called right $\Sres$-resolvent operator. Formula \eqref{mix res eq} involves a nice interaction between the left and right $\Sres$-resolvent operators and allows some properties to be deduced, but it does not provide a tool for obtaining an expansion of $\Sq(\A)$ analogous to \eqref{C-exp of R}.

In the present paper, we prove the following resolvent equation for $\Sq$ that involves only the left $\Sres$-resolvent operator and reduces to \eqref{C-exp of R} when $q$ is real:
\begin{equation}\label{eq:qre-intro}
  \Sp(\A) - \Sq(\A) 
    = \Q_q(\A)(q-p)+\Q_q(\A)\Sp(\A)\triangle_q(p),
\end{equation}
where $\triangle_q:\quat\lra\quat$ is the function defined by
\begin{equation}
  \triangle_{q}(p) := p^2 - 2\re(q)p + |q|^2.
\end{equation}
The proof of identity \eqref{eq:qre-intro} will be based on the following resolvent equation for $\Q_q(\A)$:
\begin{equation}\label{resolv.eq. for Q-intro}
   \Q_p(\A) - \Q_q(\A) 
    = \big(\Delta_p(\A) - \Delta_q(\A)\big) \Q_p(\A) \Q_q(\A), 
\end{equation}         
which enjoys a formal symmetry that has a strong similarity with the classical resolvent equation \eqref{res eq class}. Our resolvent equation \eqref{eq:qre-intro} is the key to establishing a new series expansion that is the quaternionic counterpart of \eqref{C-exp of R}. Indeed, we will show that for $q$ in a suitable open neighborhood of $q_0$ we have
\begin{equation}\label{S-exp-intro}
  \Sq(\A) = \sum_{n=1}^\infty \C_{n+1}\qspow(q),
\end{equation}
with coefficients
\begin{equation}\label{Bn-intro}
  \C_n :=
  \begin{cases}
    \Q_{q_0}(\A)^k & \text{if $n=2k$, \  $k \in \enne$,} \\
    \Q_{q_0}(\A)^k{\S^{{}^{-1}}_{q_0}}(\A) & \text{if $n=2k+1$, \  $k \in \enne$,}
  \end{cases}
\end{equation}
and ``powers''
\begin{equation}\label{pn-intro}
  \qspow(q) :=
  \begin{cases}
    \triangle_{q_0}(q)^k & \text{if $n=2k$, \ $k \in \enne$,} \\
    (q-q_0)\triangle_{q_0}(q)^k & \text{if $n=2k+1$, \ $k \in \enne$.}
  \end{cases}
\end{equation}
Although the  expressions \eqref{Bn-intro} and \eqref{pn-intro} may seem complicated at first glance, one can be convinced of the naturalness and relevance of the series expansion \eqref{resolv.eq. for Q-intro} in the quaternionic framework when considering the case of functions. Indeed, in \cite{Sto12}, C.\ Stoppato proved that every right slice regular function $f:\quat\funzione\quat$ admits a series expansion of the form
\begin{equation}\label{series of functions}
  f(q) = \sum_{n=0}^\infty b_n\qspow(q)
\end{equation}
in an open neighborhood of $q_0$ in $\quat$, for suitable constants $b_n \in \quat$. The importance of \eqref{series of functions} is clear if we consider the fact that in the literature the previous existing series expansions of $f$ converge in sets which are not open in $\quat$ (see the discussion in \cite[Chapter 8]{GenSto22}).

We would also like to emphasize that the way we prove \eqref{S-exp-intro} is completely independent of the proof of \eqref{series of functions} in \cite{Sto12}. We have thus found an alternative description of the role played by functions $\qspow$ in quaternionic analysis. This description is quite surprising and unexpected, considering that it does not seem possible to deduce \eqref{S-exp-intro} directly from \eqref{series of functions}.

In the next section, we state our main results in their complete form. In Section \ref{S:proof of the resolvent equation}, we prove our resolvent equation. Section \ref{S:Proof of of the expansion} is devoted to the proofs of the series expansion of the $\Sres$-resolvent operator and some of its corollaries. In these proofs, we use some arguments from the theory of real analytic functions. In the final Appendix \ref{appendix}, we prove some auxiliary facts about vector-valued analytic functions.


\section{Preliminaries and main results}\label{S:preliminaries}

The symbol $\enne$ will denote the set of integers greater than or equal to $0$.

\begin{Def}
We denote by $\quat$ the real algebra $\erre^4$, generated by the standard basis $1 := (1,0,0,0), i := (0,1,0,0), j := (0,0,1,0), k := (0,0,0,1)$, endowed with the only associative multiplication such that $i^2 = j^2 = k^2 = -1$ and $ij = k, jk=i, ki= j$. The elements of 
$\quat$ are called \emph{quaternions}, its identity is $1_\quat = 1$ so that every $q \in \quat$ can be written in the form 
\begin{equation}
  q = x_0 + x_1i + x_2j + x_3k \in \quat, \quad x_0, x_1, x_2, x_3 \in \erre.
\end{equation}
We define the \emph{conjugate of $q$} by
$\overline{q} := x_0 - x_1i - x_2j - x_3k,$
and we set $\re(q) := x_0$, $\im(q) := x_1i + x_2j + x_3k$. We will endow $\quat$ with the topology induced by the Euclidean norm defined by $|q| := \sqrt{q\qbar}$ for $q \in \quat$.
The \emph{unit imaginary sphere} is the set defined by
\begin{equation}
  \mathbb{S} := \{\j \in \quat:\j^2 = -1\}.
\end{equation}
For every $\j \in \mathbb{S}$, we set $\ci_\j := \{r+s\j \in \quat:r, s \in \erre\}$ so that 
$\quat = \bigcup_{\j \in \mathbb{S}} \ci_\j$.
The real vector subspace of $\quat$ generated by $1$ will be identified with $\erre$.
\end{Def}

\begin{Def}
A \emph{two-sided $\quat$-module} is a commutative group $(X,+)$ endowed with a left scalar multiplication $\quat \times X \funzione X : (q,x) \longmapsto qx$ and a right scalar multiplication 
$X \times \quat \funzione X : (x,q) \longmapsto xq$ such that 
\begin{alignat}{5}
 & q(x+y) = qx + qy, &\quad& (x+y)q = xq + yq	&\qquad&\forall x, y \in X,	&\quad &\forall q \in \quat, \notag \\
 & (p+q)x = px + qx, &\quad& x(p+q) = xp + xq	&\qquad& \forall x \in X,     &\quad &\forall p, q \in \quat, \notag \\   
 & 1x = x 	= x1	       &\quad&                              &\qquad& \forall x \in X,     &                                       \notag \\  
 & p(qx) = (pq)x,      &\quad& (xp)q = x(pq)  	&\qquad&    \forall x \in X, 	&\quad &\forall p, q \in \quat, \notag \\
 & p(xq) = (p x)q      &\quad& \qquad                 &\qquad&        \forall x \in X, &\quad &\forall p, q \in \quat,   \notag \\
 & r x = x r 	      &\quad&  \qquad                 &\qquad&   \forall x \in X, &\quad &\forall r \in \erre.  \notag
\end{alignat}
If $Y$ is a subgroup of $X$, then $Y$ is called a \emph{left $\quat$-submodule} if $qx \in Y$ whenever $x \in Y$ and $q \in \quat$. Instead, $Y$ is called a \emph{right $\quat$-submodule} of $X$ if $xq \in Y$ whenever $x \in Y$ and $q \in \quat$. Finally, $Y$ is called a \emph{two-sided $\quat$-submodule} of $X$ if it is both a left and a right $\quat$-submodule of $X$.
\end{Def}

The previous definition (except the condition involving the real scalar $r$) can be found in 
\cite[Chapter 1, Section 2]{AndFul74}. 

\begin{Def}
A two-sided $\quat$-module $X$ is called a \emph{normed two-sided $\quat$-module} if it is endowed with a \emph{$\quat$-norm on $X$}, that is, a function $\norma{\cdot}{} : X  \funzione \clsxint{0,\infty}$ such that
\begin{alignat}{3}
  & \norma{x}{} = 0 \ \ \Longrightarrow \ \ x = 0, \notag \\
  & \norma{x + y}{} \le \norma{x}{} + \norma{y}{} 		& \qquad & \forall x, y \in X,  \notag \\
  & \norma{qx}{} = \norma{xq}{} = |q|\norma{x}{} 	
  & \qquad & \forall x \in X, \quad \forall q \in \quat. \label{eq:homog} 
\end{alignat}
We say that $X$ is a \emph{Banach two-sided $\quat$-module} if the metric 
$d : X \times X \funzione \clsxint{0,\infty} : (x,y) \longmapsto \norma{x-y}{}$ is complete.
\end{Def}

\begin{Def}
Assume that $X$ is a two-sided $\quat$-module. Let $D(\A)$ be a right $\quat$-submodule of $X$. We say that $\A : D(\A) \funzione X$ is \emph{right linear} if it is additive and
\begin{equation*}
 \A(xq) = \A(x) q \qquad \forall x \in D(\A), \quad \forall q \in \quat.
\end{equation*}
As usual, the notation $\A x$ is often used in place of $\A(x)$. The identity operator $\Id_X$, or simply $\Id$, is right linear with $D(\!\!\hspace{0.32ex}\Id) = X$. Moreover, if $X$ is a normed two-sided $\quat$-module, then we say that 
$\A : D(\A) \funzione X$ is \emph{closed} if its graph is closed in $X \times X$. As in the classical theory, we set $D(\A^n) := \{x \in D(\A^{n-1}) \, : \, \A^{n-1} x \in D(\A)\}$ for every $n \in \enne \setmeno \{0\}$, where $\A^0:=\Id$, $\A^n x := \A\A^{n-1} x$ for $x \in D(\A^n)$.
\end{Def}

For the first part of the previous definition, see also \cite[Chapter 1, p.\ 42]{AndFul74}. For the following notions, see, e.g., \cite[Chapter 1, pp.\ 55-57]{AndFul74}.

\begin{Def}
Let $X$ be a two-sided $\quat$-module, let $D(\A)$ be a right $\quat$-submodule of $X$ and let 
$q \in \quat$. If $\A : D(\A) \funzione X$ is a right linear operator, then we define the right linear operator 
$q \A : D(\A) \funzione X$ by setting
\begin{equation}
   (q \A)(x) := q \A(x), \qquad 
   x \in D(\A).  \label{operatore per scalare a sx}
\end{equation}  
If $D(\A)$ is also a left $\quat$-submodule of $X$, then we can define the right linear operator $\A q :D(\A) \funzione X$ by setting
\begin{equation}
(\A q)(x):=\A(q x), \qquad
x \in D(\A). \label{operatore per scalare a dx}
\end{equation} 
The sum of operators is defined in the usual way.
\end{Def}

\begin{Def}
Assume that the two-sided $\quat$-module $X$ is normed with $\quat$-norm $\norma{\cdot}{}$. For every 
right linear $\B : X \funzione X$, we set
\begin{equation}\label{norma operatoriale}
  \norma{\B}{} := \sup\{\norma{\B x}{}{}:\norma{x}{}\le 1\}
\end{equation}
and we define the set
$\Lin^\r(X) := \{\B  : X \funzione X:\B \text{ right linear},\ \norma{\B}{} < \infty\}$, which turns out to be a normed two-sided 
$\quat$-module when it is endowed with the usual sum of operators, the products 
\eqref{operatore per scalare a sx}, \eqref{operatore per scalare a dx}, and the norm 
\eqref{norma operatoriale}.
\end{Def}

In the reminder of the paper, we will assume that
\begin{equation}\label{X}
  \text{\it $X$ is a Banach two-sided $\quat$-module.} 
\end{equation}
In the next definition, we define the quaternionic spectral notions introduced in \cite{CoGeSaSt07} and \cite{CoGeSaSt10}.

\begin{Def}\label{D:spherical spectral notions}
Assume that \eqref{X} holds, let $D(\A)$ be a right $\quat$-submodule of $X$, and let $\A : D(\A) \funzione X$ be a closed right linear operator. Given $q \in \quat$, we define the right linear operator 
$\Delta_q(\A) : D(\A^2) \funzione X$ by setting
\begin{equation}\label{def delta}
  \Delta_q(\A) := \A^2 - 2\re (q) \,\A + |q|^2 \Id.
\end{equation}
The \emph{$\Sres$-resolvent set $\rho_\s(\A)$ of $\A$} and the 
\emph{$\Sres$-spectrum $\sigma_\s(\A)$ of $\A$} are the 
subsets of $\quat$ defined as follows: 
\begin{align}
  & \rho_\s(\A)  :=  \{q \in \quat : \text{$\Delta_q(\A)$ is bijective, $\Delta_q(\A)^{-1} \in \Lin^\r(X)$}\}, \\
  & \sigma_\s(\A)  :=  \quat \setmeno \rho_\s(\A).
\end{align}
For every $q \in \rho_\s(\A)$, the operators $\Q_q(\A) : X \funzione X$ and 
$\Sq(\A) : X \funzione X$ are defined by
\begin{align}
  & \Q_q(\A) := \Delta_q(\A)^{-1},  \\
& \Sq(\A) := \Q_q(\A) \overline{q} - \A \Q_q(\A). \label{resolvent operator}
\end{align}
The operator $\Sq(\A)$ is called \emph{$\Sres$-resolvent operator of $\A$ at $q$}.
\end{Def}

The symbol $\Sq(\A)$ and the names $\Sres$-resolvent and $\Sres$-spectrum were introduced in \cite{CoGeSaSt07,CoGeSaSt10}.
In the more general setting of alternative *-algebras, we used the notation $\mathsf{C}_q(\A)$, inspired by the Cauchy kernel function $C_q$ used in \cite{GhiPer11} for the Cauchy integral formula, see \cite{GhiRec16, GhiRec18, GhiRec23}.

We are now ready to state our main results.

Given any $q\in\quat$, $\triangle_q:\quat\funzione\quat$ denotes 
the function defined by
\begin{equation}\label{triangle_q}
\triangle_q(p):=p^2-2\re(q)p+|q|^2.
\end{equation}

\begin{Thm}[Resolvent equation]\label{T:resolv.eq.}
Assume that \eqref{X} holds. If $D(\A)$ is a right $\quat$-submodule of $X$ and $\A : D(\A) \funzione X$ is a closed right linear operator, then $\Sq(\A) \in \Lin^\r(X)$ for every $q\in\rho_\s(\A)$ and
\begin{align} \label{eq:qre}
  \Sp(\A) - \Sq(\A) 
    & = \Q_q(\A)(q-p)+\Q_q(\A)\Sp(\A)\triangle_q(p)
    \qquad \forall p, q \in \rho_\s(\A).
\end{align}
\end{Thm}

Recall that the function $\mr{u}:\quat\times\quat\lra\erre$, defined by $\mr{u}(p,q):=\sqrt{|\triangle_q(p)|}$, is a pseudo-metric on~$\quat$, called \emph{Cassini pseudo-metric}, see \cite[pp.\ 505-507]{GhiPer14} and also \cite[Section~2]{Sto12} and \cite[Section~6]{GhiPerSto17}. Given $q_0\in\quat$ and a positive real number $R$, we denote $\mr{U}(q_0,R)$ the $\mr{u}$-ball of $\quat$ centered at $q_0$ with radius $R$, that is,
\begin{equation}\label{cassini}
\mr{U}(q_0,R):=\{p\in\quat:\mr{u}(p,q_0)<R\}.
\end{equation}
The set $\mr{U}(q_0,R)$ is called \emph{Cassini ball of $\quat$ centered at $q_0$ with radius $R$}. We call \emph{Cassini topology of $\quat$} the topology of $\quat$ induced by the Cassini pseudo-metric $\mr{u}$. The origin of this nomenclature is as follows. Write $q_0=r_0+Is_0$ for some $r_0,s_0\in\erre$ and $I\in\mathbb{S}$, and set $z_0:=r_0+is_0\in\ci$ and $C:=\{z\in\ci:|(z-z_0)(z-\overline{z_0})|<R^2\}$. Observe that the boundary of $C$ in $\ci$ is a Cassini oval for $z_0\not\in\erre$ and a circle for $z_0\in\erre$. It is immediate to verify that $\mr{U}(q_0,R)$ is the circularization of $C$, that is, $\mr{U}(q_0,R)=\bigcup_{r,s\in\erre,\,r+is\in C}(r+s\mathbb{S})$. Consequently, $\mr{U}(q_0,R)$ is open with respect to the Euclidean topology of $\quat=\erre^4$. It follows that the Cassini topology of $\quat$ is coarser than the Euclidean topology of $\quat$.

\begin{Thm}[Series expansion of the $\Sres$-resolvent operator]\label{T:main}
Assume that \eqref{X} holds. Consider a right $\quat$-submodule $D(\A)$ of $X$ and a closed right linear operator $\A : D(\A) \funzione X$ such that $\rho_\s(\A) \cap \erre \neq \vuoto$. Pick a point $q_0\in\rho_\s(\A)$, set $R:=\norma{\Q_{q_0}(\A)}{}^{-1/2}>0$, and define the operators $\C_n\in\Lin^\r(X)$ and the functions $\qspow:\quat\lra\quat$ by
\begin{equation}\label{Bn first}
  \C_n :=
  \begin{cases}
    \Q_{q_0}(\A)^k & \text{if $n=2k$, \  $k \in \enne$,} \\
    \Q_{q_0}(\A)^k{\S_{q_0}^{{}^{-1}}}(\A) & \text{if $n=2k+1$, \  $k \in \enne$}
  \end{cases}
\notag\\
\end{equation}
and
\begin{equation*}
  \qspow(q) :=
  \begin{cases}
    \triangle_{q_0}(q)^k & \text{if $n=2k$, \ $k \in \enne$,} \\
    (q-q_0)\triangle_{q_0}(q)^k & \text{if $n=2k+1$, \ $k \in \enne$.}
  \end{cases}
\end{equation*}
where $\triangle_{q_0}(q)$ is defined by \eqref{triangle_q}.
Then we have:
\begin{itemize}
 \item[$(\mr{i})$] $\mr{U}(q_0,R)\subset\rho_\s(\A)$.
 \item[$(\mr{ii})$] $\rho_\s(\A)$ is open in the Cassini topology of $\quat$ (hence also in the Euclidean topology of $\quat$).
 \item[$(\mr{iii})$] For every $q\in\mr{U}(q_0,R)$, the series $\sum_{n=0}^{\infty}(-1)^n \C_{n+1} \qspow(q)$ converges in $\Lin^r(X)$ and its sum is $\Sq(\A)$, that is,
\begin{equation*}\label{series of S}
\Sq(\A)=\sum_{n=0}^{\infty}(-1)^n \C_{n+1} \qspow(q)\,.
\end{equation*}
\end{itemize}
\end{Thm}

Let us present three corollaries.

Following \cite[Definition 6]{GhiPer11}, we define the \emph{spherical derivative} $\partial_\s\qspow(q):\quat\lra\quat$ of $\qspow$ as $\partial_\s\qspow(q):=(\qspow(q) - \qspow(\qbar))(q-\qbar)^{-1}$ if $q\in\quat\setmeno\erre$ and $\partial_\s\qspow(q):=\frac{\partial\qspow}{\partial r}(q)$ if $q\in\erre$, where $\frac{\partial\qspow}{\partial r}$ denotes the derivative of $\qspow$ in the direction of the real axis $\erre$.

\begin{Cor}\label{cor3}
Assume that \eqref{X} holds. Let $D(\A)$ be a right $\quat$-submodule of $X$, let $\A : D(\A) \funzione X$ be a closed right linear operator such that $\rho_\s(\A) \cap \erre \neq \vuoto$, let $q_0\in \rho_\s(\A)$ and let $R:=\norma{\Q_{q_0}(\A)}{}^{-1/2}>0$. Then, for every $q\in\mr{U}(q_0,R)\subset\rho_\s(\A)$, the series $\sum_{n=0}^{\infty}(-1)^{n+1} \C_{n+1}\partial_\s\qspow(q)$ converges in $\Lin^r(X)$ and its sum is $\Q_q(\A)$, that is,
\[
\Q_q(\A)=\sum_{n=0}^{\infty}(-1)^{n+1} \C_{n+1} \partial_\s \qspow(q)\,.
\]

Moreover, if we endow $\rho_\s(\A)$ with the relative topology induced by the Cassini topology of $\quat$, then the mapping $\rho_\s(\A)\funzione\Lin^r(X):q\longmapsto \Q_q(\A)$ is continuous.
\end{Cor}

Given $q_0\in\quat$ and any subset $T$ of $\quat$, we set 
$\mr{u}(q_0,T):=\inf_{p\in T}\mr{u}(p,q_0)$ if $T\neq\vuoto$ and $\mr{u}(q_0,T):=+\infty$ if $T=\vuoto$.

\begin{Cor}\label{cor1}
Assume that \eqref{X} holds. Let $D(\A)$ be a right $\quat$-submodule of $X$ and let $\A : D(\A) \funzione X$ be a closed right linear operator such that $\rho_\s(\A) \cap \erre \neq \vuoto$.
If $q_0 \in \rho_\s(\A)$, then
\[
  \mr{u}(q_0,\sigma_\s(\A)) \ge \frac{1}{\norma{\Q_{q_0}(\A)}{}^{1/2}}\,.
\]
\end{Cor}

\begin{Cor}\label{cor2}
Assume that \eqref{X} holds. Let $D(\A)$ be a right $\quat$-submodule of $X$, let $\A : D(\A) \funzione X$ be a closed right linear operator such that $\rho_\s(\A) \cap \erre \neq \vuoto$, let $q_0\in\quat$ and let 
$(p_n)_{n\in\mathbb{N}}$ 
be a sequence in $\rho_\s(\A)$ converging to $q_0$ with respect to the Cassini topology of $\quat$. Then $q_0 \in \sigma_\s(\A)$ if and only if $\lim_{n\to\infty}\norma{\Q_{p_n}(\A)}{} = \infty$.
\end{Cor}


\section{Proof of Theorem \ref{T:resolv.eq.}}\label{S:proof of the resolvent equation}

We start by proving a resolvent equation for the operators $\Q_q(\A)$.

\begin{Prop}[Resolvent equation for $\Q_q(\A)$]\label{P:resolv.eq. for Q}
Assume that \eqref{X} holds. If $D(\A)$ is a right $\quat$-submodule of $X$ and $\A : D(\A) \funzione X$ is a closed right linear operator, then
\begin{align}
   \Q_p(\A) - \Q_q(\A) 
   & = \big(\Delta_p(\A) - \Delta_q(\A)\big) \Q_p(\A) \Q_q(\A) \notag \\
   & = \big(\Delta_p(\A) - \Delta_q(\A)\big) \Q_q(\A) \Q_p(\A) 
          \qquad \forall p, q \in \rho_\s(\A). \label{resolv.eq. for Q}
\end{align}
\end{Prop}

\begin{proof}
In order to shorten the notation, we 
omit the dependence on $\A$ of the various operators. We have that
\begin{align}\label{resolv.eq. for Q-a}
  \Q_p - \Q_q 
    & = \Q_p - \Delta_p \Q_p \Q_q \notag \\
    & = \Q_p - (\Delta_p - \Delta_q + \Delta_q) \Q_p \Q_q \notag \\ 
    & = \Q_p - (\Delta_p - \Delta_q)\Q_p \Q_q - \Delta_q \Q_p \Q_q. 
\end{align}
By definition \eqref{def delta}, we have that $\Delta_p$ and $\Delta_q$ commute, therefore $\Q_p\Q_q = \Q_q\Q_p$ and from \eqref{resolv.eq. for Q-a} we infer that
\begin{align}
  \Q_p - \Q_q 
    & = \Q_p - (\Delta_p - \Delta_q)\Q_p \Q_q - \Delta_q \Q_q \Q_p \notag \\
    & = \Q_p - (\Delta_p - \Delta_q)\Q_p \Q_q - \Q_p \notag \\
    & = (\Delta_q - \Delta_p)\Q_p \Q_q.  \notag
\end{align}
This completes the proof.
\end{proof}

\begin{Lem}\label{L:Cq-Cp-step1}
Assume that \eqref{X} holds. If $D(\A)$ is a right $\quat$-submodule of $X$ and $\A : D(\A) \funzione X$ is a closed right linear operator, then
\begin{equation}\label{Cq-Cp-step1}
  \Sp(\A) - \Sq(\A) =  \Q_q(\A)\Big[(\pbar - \qbar)\Id + \big(\Delta_q(\A) - \Delta_p(\A)\big)\Sp(\A)\Big] 
  \qquad \forall p, q \in \rho_\s(\A).
\end{equation}
\end{Lem}

\begin{proof}
We set $\Delta_p := \Delta_p(\A)$, $\Sp := \Sp(\A)$, 
$\Delta_q := \Delta_q(\A)$ and $\Sq := \Sq(\A)$. Starting from definition \eqref{resolvent operator} and using \eqref{resolv.eq. for Q}, we infer that 
\begin{align}
  \Sp - \Sq 
    & = (\Q_p\pbar - \A\Q_p) - (\Q_q\qbar - \A\Q_q) \notag \\
    & = \Q_p\pbar - \Q_q\qbar - \A(\Q_p - \Q_q) \notag \\
    & = \Q_p\pbar - \Q_q\pbar + \Q_q\pbar - \Q_q\qbar - \A(\Q_p - \Q_q) \notag \\
    & = (\Q_p - \Q_q)\pbar + \Q_q(\pbar - \qbar) - \A(\Q_p - \Q_q) \notag \\
    & = (\Delta_p - \Delta_q)\Q_p \Q_q\pbar + \Q_q(\pbar - \qbar) - \A(\Delta_p - \Delta_q)\Q_p \Q_q\,. 
 \label{Cq-Cp-step1-a}
\end{align}
The operators $\Q_q$ and $\A$ commute on $D(\A)$ for every $q \in \rho_\s(A)$ (cf., e.g., \cite[Proposition 2.28]{GhiRec16}, where it is assumed the non necessary hypothesis that $D(\A)$ is dense in $X$), hence $\Q_q$ and $\Delta_q$ commute on $D(\A^2)$. Therefore, using Proposition \ref{P:resolv.eq. for Q} and \eqref{Cq-Cp-step1-a}, we obtain:
\begin{align}
  \Sp - \Sq 
    & = (\Delta_p - \Delta_q)\Q_q \Q_p\pbar + \Q_q(\pbar - \qbar) - \A(\Delta_p - \Delta_q)\Q_q \Q_p 
           \notag \\
    & = \Q_q(\Delta_q - \Delta_p) \Q_p\pbar  + \Q_q(\pbar - \qbar) - \Q_q(\Delta_q - \Delta_p) \A \Q_p 
          \notag \\
    & = \Q_q(\pbar - \qbar) + \Q_q(\Delta_q - \Delta_p)(\Q_p\pbar - \A\Q_p) \notag \\
    & = \Q_q[(\pbar - \qbar)\Id + (\Delta_q - \Delta_p)\Sp]\,, \notag
\end{align}
as required.
\end{proof}

\begin{Lem}\label{L:(Dq - Dp)Cp=}
Assume that \eqref{X} holds. If $D(\A)$ is a right $\quat$-submodule of $X$ and $\A : D(\A) \funzione X$ is a closed right linear operator, then
\begin{equation}\label{(Dq - Dp)Cp=}
  (\pbar - \qbar)\Id + \big(\Delta_q(\A) - \Delta_p(\A)\big)\Sp(\A) = 
  (q - p)\Id + \Sp(\A)\triangle_q(p)
  \quad \forall p, q \in \rho_\s(\A).
\end{equation}
\end{Lem}

\begin{proof}
We set $\Delta_p := \Delta_p(\A)$, $\Sp := \Sp(\A)$, 
$\Delta_q := \Delta_q(\A)$ and $\Sq := \Sq(\A)$. We have:
\begin{align}
  & (\pbar - \qbar)\Id + (\Delta_q - \Delta_p)\Sp \notag \\
    & = (\pbar - \qbar)\Id + [2\re(p-q)\A + (|q|^2 - |p|^2)\Id]\Sp \notag \\
    & = (\pbar - \qbar)\Id + 2\re(p-q)\A \Sp + (|q|^2 - |p|^2)\Sp. \label{(Dq - Dp)Cp=-a}
\end{align}
Since 
\[
\A\Sp = \Sp p-\Id
\]
(cf., e.g., \cite[Lemma 2.29]{GhiRec16}, where it is supposed the non necessary assumption that $D(\A)$ is dense in $X$), \eqref{(Dq - Dp)Cp=-a} implies
\begin{align}
  & (\pbar - \qbar)\Id + (\Delta_q - \Delta_p)\Sp \notag \\
    & = (\pbar - \qbar)\Id + 2\re(p-q)(\Sp p - \Id) + (|q|^2 - |p|^2)\Sp \notag \\
    & = (\pbar - \qbar - 2\re(p-q))\Id + 2\re(p-q)\Sp p + (|q|^2 - |p|^2)\Sp \notag \\
    & = (\pbar - \qbar - 2\re(p-q))\Id + \Sp(2\re(p-q)p + |q|^2 - |p|^2) \notag \\
    & = (q - p)\Id + \Sp(2\re(p)p-2\re(q)p + |q|^2 - |p|^2) \notag \\
    & = (q - p)\Id + \Sp(p^2+|p|^2 -2\re(q)p+ |q|^2 - |p|^2) \notag \\
    & = (q - p)\Id + \Sp(p^2 -2\re(q)p+ |q|^2)=(q - p)\Id + \Sp\triangle_q(p),\notag
\end{align}
as required.
\end{proof}

\begin{Lem}\label{AQ Lin}
$\A\Q_q(\A) \in \Lin^\r(X)$ for every $q \in \rho_\s(\A)$.
\end{Lem}

\begin{proof}
By definition, we have that $\Q_q(\A) \in \Lin^\r(X)$. Endow $(X,+)$ with the real scalar multiplication 
$\erre \times X \funzione X : (r,x) \longmapsto rx = xr$. Thanks to \eqref{eq:homog}, $X$ can be regarded as a real Banach space, so that $\A$ is a closed $\erre$-linear operator on it and an application of the closed graph theorem implies that the right linear operator $\A\Q_q(\A)$ is continuous.
\end{proof}

We conclude this section with the proof of our first main result.

\begin{proof}[Proof of Theorem \ref{T:resolv.eq.}]
Equation \eqref{eq:qre} immediately follows by combining Lemmas \ref{L:Cq-Cp-step1} and \ref{L:(Dq - Dp)Cp=}. Moreover, by Lemma \ref{AQ Lin} and definition \eqref{resolvent operator}, we deduce that $\Sq(\A) \in \Lin^\r(X)$.
\end{proof}


\section{Proof of Theorem \ref{T:main} and its corollaries} \label{S:Proof of of the expansion}

\emph{Throughout this section, we assume that \eqref{X} holds.}

Let us start by recalling the notion of vector-valued right slice regular function in the sense of \cite[Definition 4.2]{GhiRec18}.

\begin{Def}\label{D:Xslice regularity}
For every $\j \in \mathbb{S}$, let $X_\j$ denote the abelian group $X$ endowed with the structure of complex Banach space induced by the operation $\ci \times X \funzione X : (z,x) \longmapsto zx$ defined by
\[
  zx := x(r+s\j) \qquad x \in X,\ z \in \ci, \ z = r+si,\ r,s \in \erre.
\]

Let $D$ be a non-empty open subset of $\ci$ invariant under the complex conjugation and let $\OO_D$ be its circularization in $\quat$, that is, $\OO_D$ is the non-empty open subset of $\quat=\erre^4$ defined by
$$
\OO_D := \{q \in \quat:q = r + s\j,\ r, s \in \erre,\ \j \in \mathbb{S},\ r+si \in D \}.
$$

A mapping $F=(F_1,F_2):D\to X^2$ is called \emph{stem function} if $F_1(r-si)=F_1(r+si)$ and $F_2(r-si)=-F_2(r+si)$ for every $r+si\in D$ with $r,s\in\erre$.

Let $f:\OO_D\funzione X$ be a function. For every $\j \in \mathbb{S}$, define the function $f_\j : D \funzione X_\j$ as $f_\j(r+si):=f(r+s\j)$ for every $r,s\in\erre$ with $r+si\in D$. The function $f:\OO_D\funzione X$ is said to be \emph{right slice regular} if there exists a stem function $F=(F_1,F_2):D\funzione X^2$ such that:
\begin{itemize}
 \item[$(\mr{i})$] $F_1$ and $F_2$ are of class $C^1$ (in the usual real 
 sense);
 \item[$(\mr{ii})$] $f_\j(z)=F_1(z)+F_2(z)\j\qquad\forall z\in D,\;\forall\j\in\mathbb{S}$;
 \item[$(\mr{iii})$] $f_\j$ is holomorphic for every $\j\in\mathbb{S}$, that is,
 \begin{equation}\label{equa1}
  \frac{\partial f_\j}{\partial r}(z) +i\frac{\partial f_\j}{\partial s}(z) =\frac{\partial f_\j}{\partial r}(z) + \frac{\partial f_\j}{\partial s}(z)\j=0\qquad\forall z\in D,\;\forall\j\in\mathbb{S}\,.
\end{equation}
As $0=\frac{\partial f_\j}{\partial r}+\frac{\partial f_\j}{\partial s}\j=\big(\frac{\partial F_1}{\partial r}-\frac{\partial F_2}{\partial s}\big)+\big(\frac{\partial F_1}{\partial s}+\frac{\partial F_2}{\partial r}\big)\j$ for every $\j\in\mathbb{S}$, the holomorphy conditions \eqref{equa1} for the $f_\j$ are equivalent to the following Cauchy-Riemann equation for~$F$:
\begin{equation}\label{equa2}
\frac{\partial F_1}{\partial r}=\frac{\partial F_2}{\partial s}
\quad\text{ and }\quad
\frac{\partial F_1}{\partial s}=-\frac{\partial F_2}{\partial r} \quad\text{on $D$}\,.
\end{equation}
\end{itemize}
In this situation, we say that $f$ is induced by $F$.
\end{Def}

\begin{Rem}\label{rem}
$(\mr{i})$ We have $f(q)=F_1(q)$ for every $q\in\OO_D\cap\erre=D\cap\erre$. Thus, if $\OO_D\cap\erre\neq\vuoto$, the restriction of $f$ to $\OO_D\cap\erre$ is of class $C^1$ and, for every $q\in\OO_D\cap\erre$, the derivative $\frac{\partial f}{\partial r}(q)$ of $f$ at $q$ in the direction of the real axis $\erre$ makes sense and equals $\frac{\partial F_1}{\partial r}(q)$.

$(\mr{ii})$ The open subset $\OO_D$ of $\quat$ is axially symmetric in the sense that, if $q=r+s\j\in\OO_D$ for some $r,s\in\erre$ and $\j\in\mathbb{S}$, then the whole $2$-sphere $r+s\mathbb{S}$ is contained in $\OO_D$. \hfill {\tiny$\blacksquare$} 
\end{Rem}

The concept of vector-valued (left) slice regularity was first introduced in 
\cite[Definition 3.4]{AlpColLewSab15} as \emph{slice hyperholomorphicity}, see also 
\cite[Theorem 3.6]{AlpColLewSab15}. Bearing in mind Proposition \ref{char-Y-holom} in the Appendix, 
the definition of vector-valued right slice hyperholomorphic function reads as follows. Let $\OO$ be a (not necessarily axially symmetric) non-empty open subset $\OO$ of $\quat$, and let $f:\OO\funzione X$ be a function. For every $\j\in\mathbb{S}$ with $\OO\cap\ci_\j\neq\vuoto$, let $\psi_\j:\ci\funzione\ci_\j$ be the field isomorphism $\psi_\j(r+is):=r+s\j$, let $D_\j:=\psi_\j^{-1}(\OO\cap\ci_\j)$ and let $f_\j:D_\j\funzione X$ be the function $f_\j(r+is):=f(\psi_\j(r+is))=f(r+s\j)$. The function $f:\OO\funzione X$ is called \emph{right slice hyperholomorphic} if, for every $\j\in\mathbb{S}$ with $\OO\cap\ci_\j\neq\vuoto$, the function $\OO\cap\ci_\j\funzione X_\j:r+s\j\longmapsto f(r+s\j)$ is holomorphic, that is, it is of class $C^1$ (in the real sense) and $\frac{\partial f_\j}{\partial r}+\frac{\partial f_\j}{\partial s}\j=0$ for every $r+is\in D_\j$ and $\j\in\mathbb{S}$. As we see in the next result, if $\OO=\OO_D$ with $D$ connected and $D\cap\erre\neq\vuoto$, the right slice hyperholomorphicity coincides with the right slice regularity defined in Definition~\ref{D:Xslice regularity}.  

\begin{Lem}\label{equivalence}
Let $D$ be a non-empty connected open subset of $\ci$ invariant under the complex conjugation such that $D\cap\erre\neq\vuoto$, and let $f:\OO_D\funzione X$. Then, $f$ is right slice regular if and only if it is right slice hyperholomorphic, that is, if and only if the only condition \eqref{equa1} of Definition \ref{D:Xslice regularity} is satisfied.
\end{Lem}

\begin{proof}
Evidently, if $f$ is right slice regular, then it is also right slice hyperholomorphic.
Let us prove the converse. Suppose that $f$ is right slice hyperholomorphic. 
Hence, for every $\j \in \su$, the function $f_\j : D \funzione X_\j$ defined by $f_\j(r+si) := f(r+s\j)$, $r, s \in \erre,\ r+si \in D$ is holomorphic and condition \eqref{equa1} of Definition \ref{D:Xslice regularity} is satisfied. Now let us fix $\k \in \mathbb{S}$ and define the functions 
$F_1^{\k} : D \funzione X$ and $F_2^{\k} : D \funzione X$ by 
\begin{equation}\textstyle
  F_1^{\k}(z) := \frac{1}{2}(f_\k(z) + f_\k(\overline{z})), \quad
   F_2^{\k}(z) := -\frac{1}{2}(f_\k(z) - f_\k(\overline{z})) \k, \qquad z \in D, \label{F_2}
\end{equation}
which are of class $C^1$ in the real sense. Define also $h^{\k} : \Omega_D \funzione X$ by 
$h^{\k}(r+s\j) := F_1^{\k}(r+si) + F_2^{\k}(r+si)\j$, for $r, s \in \erre,\ \j \in \mathbb{S}$, $r+s\j \in \Omega_D
$. We will prove that $F_1^{\k}$, $F_2^{\k}$, and $h^\k$ do not depend on $\k$. For every $\j \in \mathbb{S}$ and for all 
$r, s \in \erre$ such that $q = r +s\j \in \Omega_D$, setting $z = r + s\mathrm{i} \in D$ and observing that 
$f_\k(z) = f_{-\k}(\zbar)$, we find that 
\begin{align}
  \frac{\partial h^\k_\j}{\partial r}(z) + \frac{\partial h^\k_\j}{\partial s}(z) \j
   & = \frac{1}{2}\frac{\partial}{\partial r} (f_\k(z)+f_\k(\zbar) - f_\k(z)\k\j + f_\k(\zbar)\k\j) \notag \\
   &    \phantom{=\ } + \frac{1}{2}\frac{\partial}{\partial s} (f_\k(z)\j+f_\k(\zbar)\j + f_\k(z)\k - f_\k(\zbar)\k) \notag \\
& = \frac{1}{2}\left(\frac{\partial f_\k}{\partial r}(z) + \frac{\partial f_\k}{\partial s}(z)\k\right)
   + \frac{1}{2}\left(\frac{\partial f_{-\k}}{\partial r}(z) +  \frac{\partial f_{-\k}}{\partial s}(z)(-\k)\right) \notag \\
   & \phantom{=\ } - \frac{1}{2}\left(\frac{\partial f_\k}{\partial r}(z) +\frac{\partial f_\k}{\partial s}(z)\k\right)\k\j        + \frac{1}{2}\left(\frac{\partial f_{-\k}}{\partial r}(z) + \frac{\partial f_{-\k}}{\partial s}(z)(-\k)\right)\k\j 
   = 0,\notag
\end{align}
because $f_\j$ is holomorphic, hence $h^\k_\j : D \funzione X_\j$ is holomorphic. Moreover, we have $h^\k_\j(r) = f(r)$ for every $r \in D \cap \erre$, so that 
by holomorphic continuation we find that $h^\k_\j(z) = f_\j(z)$ for every $z \in D$. Hence, by the arbitrariness of $\j$, we infer that $f = h^\k$ on $\Omega_D$, and we have proved that
\begin{equation}
  f(r+s\j) = F^\k_1(r+s\mathrm{i}) + F^\k_2(r+s\mathrm{i})\j \qquad 
  \forall r, s \in \erre,\ \j \in \mathbb{S},\ r+s\j \in \Omega_D.
\end{equation}
In particular $F_1^{\k}$, $F_2^{\k}$ and $h^\k$ do not depend on the choice of 
$\k \in \mathbb{S}$ so we can adopt the notations $F_1 := F_1^{\k}$, $F_2 := F_2^{\k}$, which are the stem functions we are looking for.
\end{proof}

Let us recall the definition of spherical derivative of a right slice regular function. To do this, we use implicitly Remark \ref{rem}$(\mr{i})$. 

\begin{Def}
Let $D$ be a non-empty open subset of $\ci$ invariant under the complex conjugation (which does not necessarily intersect the real axis $\erre$) and let $f:\OO_D \lra X$ be a right slice regular function. We define the \emph{spherical derivative $\partial_\s f:\OO_D\funzione X$ of $f$} as follows:
\begin{equation}\label{def g}
  \partial_\s f(q) := 
  \begin{cases}
  (f(q)-f(\overline{q}))(q-\overline{q})^{-1} & \text{ if $q \in \OO_D \setmeno \erre$},  \\
  \ \\
  \textstyle \frac{\partial f}{\partial r}(q) & \text{ if $q \in \OO_D \cap \erre$.}
\end{cases}
\end{equation}
\end{Def}

In the next definition, we specify what we mean by a real analytic function with values in a real Banach space.

\begin{Def}\label{real analytic}
Let $n \in \enne\setmeno\{0\}$, let $\OO$ be a non-empty open subset of $\erre^n$, let 
$(Z,\norma{\cdot}{})$ be a real Banach space and let $f : \Omega \funzione Z$ be a function. We say that 
$f$ is \emph{real analytic} if for every $w_0 \in \OO$, there exist an open neighborhood $V_0$ of $w_0$ in $\OO$ and elements $c_h \in Z$ for every $h\in\enne^n$ such that $V_0\subset\OO$ and 
\begin{equation}\label{Real power series}
f(w)=\sum_{h\in\enne^n}c_h(w-w_0)^h \qquad \forall w \in V_0,
\end{equation}
where $v^\ell := v_1^{\ell_1}\cdots v_n^{\ell_n}$ for every $v = (v_1,\ldots,v_n) \in \erre^n$ and $\ell= (\ell_1,\ldots,\ell_n) \in \enne^n$. The convergence in \eqref{Real power series} is meant in the sense of summability, that is, for every $\eps > 0$, there exists a finite set $F_\eps \subset \enne^n$ such that $\norma{f(w) - \sum_{h \in F} c_h (w-w_0)^h}{} < \eps$ for every finite set $F\subset\enne^n$ with $F_\eps \subset F$.
\end{Def}

In the following lemma, we prove two facts about vector-valued real analytic functions that will be crucial in our proofs.

\begin{Lem}\label{vector analytic}
Let $n \in \enne\setmeno\{0\}$, let $\OO$ be a non-empty open subset of $\erre^n$ and let $(Z,\norma{\cdot}{})$ be a real Banach space. We have:
\begin{itemize}
\item[(i)] Let $f:\OO\funzione Z$ be a real analytic function and let $V_f:=\{w \in \OO:f(w) = 0\}$ be the vanishing set of $f$. If $\OO$ is connected and $V_f$ has an interior point in~$\OO$, then $V_f=\OO$, that is, $f(w) = 0$ for every $w\in\OO$. 
\item[(ii)] Pick $c_h \in Z$ for every $h \in \enne^n$ in such a way that the series $\textstyle\sum_{h\in\enne^n}c_hw^h$ is absolutely convergent for every $w \in \Omega$. Assume that $\OO$ is stable under reflections, that is,
\begin{equation}\label{CCC}
(w_1,\ldots,w_n)\in\OO,i\in\{1,\ldots,n\}\Longrightarrow(w_1,\ldots,w_{i-1},-w_i,w_{i+1},\ldots,w_n)\in\OO.
\end{equation}
Then, the vector-valued function $g:\OO\funzione Z$, given by $g(w) := \sum_{h\in\enne^n}c_hw^h$, is real analytic.
\end{itemize}
\end{Lem}
\begin{proof}
(i) By \eqref{Real power series}, for every linear continuous $\Phi: Z \funzione \erre$, the real-valued function $\Phi \circ f : \OO\funzione\erre$ is real analytic. Moreover, $(\Phi \circ f)(w) = 0$ for all $w \in V_f$. By the analytic continuation theorem for real analytic functions with values in $\erre$, we have that $\Phi \circ f$ vanishes on the whole $\OO$. By the Hahn-Banach theorem, we deduce that $f$ vanishes on $\OO$ as well.

(ii) For every $w=(w_1,\ldots,w_n)\in\erre^n$, we set $\widehat{w}:=(|w_1|,\ldots,|w_n|)\in\erre^n$. Let $u=(u_1,\ldots,u_n)$ $\in$ $\OO$. By \eqref{CCC}, we have that $\widehat{u}\in\OO$. Therefore, there exists an open neighborhood $V$ of $0$ in $\erre^n$ so small that $\widehat{u}+\widehat{v}\in\OO$ for every $v\in V$. Recall that, given any $h=(h_1,\ldots,h_n)\in\enne^n$ and $w,z\in\erre^n$, it holds: $(w+z)^h=\sum_{k,h\in\enne^n,k\leq h}{h\choose k}w^{h-k}z^k$, where $k=(k_1,\ldots,k_n)\leq h$ means that $k_i\leq h_i$ for every $i\in\{1,\ldots,n\}$, and ${h\choose k}:=\prod_{i=1}^n{h_i\choose k_i}$. As we said, for every $v\in V$, we have $\widehat{u}+\widehat{v}\in\OO$ so the series $\sum_{h\in\enne^n}\norma{c_h}{}(\widehat{u}+\widehat{v})^h$ converges. Thus, by assumption, the series $\sum_{k,h\in\enne^n,k\leq h}c_h{h\choose k}u^{h-k}v^k$  converges absolutely:
\begin{align}
\sum_{k,h\in\enne^n,\ k\leq h}\left\|c_h{h\choose k}u^{h-k}v^k\right\|&=
\sum_{k,h\in\enne^n,\ k\leq h}\norma{c_h}{}{h\choose k}\widehat{u}^{h-k}\widehat{v}^k\notag\\
&=\sum_{h\in\enne^n}\norma{c_h}{}\sum_{k\in\enne^n,\ k\leq h}{h\choose k}\widehat{u}^{h-k}\widehat{v}^k=\sum_{h\in\enne^n}\norma{c_h}{}(\widehat{u}+\widehat{v})^h<+\infty.\label{L}
\end{align}
Fix $\omega=(\omega_1,\ldots,\omega_n)\in V$ with $\omega_i>0$ for every $i\in\{1,\ldots,n\}$. Bearing in mind \eqref{L}, for any fixed $k\in\enne^n$, we obtain:
\begin{align*}
\sum_{h\in\enne^n,\ k\leq h}\left\|c_h{h\choose k}u^{h-k}\right\|&=
\omega^{-k}\sum_{h\in\enne^n,\ k\leq h}\left\|c_h{h\choose k}u^{h-k}\omega^k\right\|\\
&\leq\omega^{-k}\sum_{l,h\in\enne^n,\ l\leq h}\left\|c_h{h\choose l}u^{h-l}\omega^l\right\|=\omega^{-k}\sum_{h\in\enne^n}\norma{c_h}{}(\widehat{u}+\omega)^h<+\infty.
\end{align*}
The latter estimate implies that, for each $k\in\enne^n$, the series $\sum_{h\in\enne^n,k\leq h}c_h{h\choose k}u^{h-k}$ converges absolutely to some $d_k\in X$. Therefore, for every $w\in u+V$, it holds:
\begin{align*}
\sum_{k\in\enne^n}d_k(w-u)^k&=\sum_{k\in\enne^n}\sum_{h\in\enne^n,k\leq h}c_h{h\choose k}u^{h-k}(w-u)^k\\
&=\sum_{h\in\enne^n}c_h\sum_{k\in\enne^n,k\leq h}{h\choose k}u^{h-k}(w-u)^k=\sum_{h\in\enne^n}c_h(u+w-u)^h=g(w).
\end{align*}
This proves that $g$ is real analytic, as required.
\end{proof}

Note that, in the previous proof of Lemma \ref{vector analytic}$(\mr{ii})$, we used only the following weak version of \eqref{CCC}: $(w_1,\ldots,w_n)\in\OO\ \Longrightarrow\ (|w_1|,\ldots,|w_n|)\in\OO$.

\begin{Rem}\label{RA-module}
The Banach two-sided $\quat$-module $X$ is also a real Banach space: it suffices to consider the scalar multiplication for real quaternions. Thus, it makes sense to speak about real analytic functions with values in $X$.

Let $C^\omega(\OO_D,\quat)$ be the ring of all quaternion-valued real analytic functions, endowed with the pointwise defined operations of addition and multiplication. Observe that, given real analytic functions $f,g:\OO_D\funzione X$ and $\eta\in C^\omega(\OO_D,\quat)$, the functions $f+g,f\eta:\OO_D\funzione\quat$ defined by $(f+g)(q):=f(q)+g(q)$ and $(f\eta)(q):=f(q)\eta(q)$ are also real analytic. It follows that the set of all real analytic functions from $\OO_D$ to $X$ is a right $C^\omega(\OO_D,\quat)$-module.  \hfill {\tiny$\blacksquare$}
\end{Rem}

Next result is a crucial tool to prove Theorem \ref{T:main}.

\begin{Prop}\label{prop:sd}
Let $D$ be a non-empty open subset of $\ci$ invariant under the complex conjugation and let $f:\OO_D \lra X$ be a function. If $f:\OO_D \lra X$ is a right slice regular function, then its spherical derivative $\partial_\s f:\OO_D\funzione X$ is a real analytic function.    
\end{Prop}

\begin{proof}
Let $F=(F_1,F_2):D\funzione X^2$ be a stem function satisfying \eqref{equa2}. Suppose that $f$ is induced by $F$. Let $q=r+s\j\in\OO_D\setmeno\erre$ with $r,s\in\erre$, $s\neq0$ and $\j\in\mathbb{S}$. Since $f(q)=F_1(r+si)+F_2(r+si)\j$ and $f(\overline{q})=F_1(r-si)+F_2(r-si)\j=F_1(r+si)-F_2(r+si)\j$, we have
\[
\partial_\s f(q)=(f(q)-f(\overline{q}))(q-\overline{q})^{-1}=(2F_2(r+si)\j)(2s\j)^{-1}=F_2(r+si)s^{-1}.
\]
Thus, it holds:
\begin{equation}\label{equa4}
\partial_\s f(q)=F_2(r+si)s^{-1}\qquad\forall q=r+s\j\in\OO_D\setmeno\erre,\; r,s\in\erre,\; s>0,\;\j\in\mathbb{S}.
\end{equation}
Fix $\k\in\mathbb{S}$. Taking Definition \ref{D:Xslice regularity} into account, we have:
\begin{equation}\label{H}
F_2(r+si)s^{-1}=((r+si)-(r-si))^{-1}(f_\k(r+si) - f_\k(r-si)) \quad \forall r+si\in D\setmeno\erre
\end{equation}
and
\begin{equation}\label{equa3}
\textstyle
F_2(r+si) = -\frac{1}{2}(f_\k(r+si) - f_\k(r-si))\k \qquad \forall r+si\in D,\;r,s\in\erre.
\end{equation}
Define the real analytic functions $\xi:\OO_D\setmeno\erre\funzione \erre$ and $\varphi,\overline{\varphi}:\OO_D\setmeno\erre\funzione D\setmeno\erre$ as
\begin{align*}
\xi(q)&:=|\im(q)|^{-1}=(x_1^2+x_2^2+x_3^2)^{-1/2}\,,\\
\varphi(q)&:=\re(q)+i|\im(q)|=x_0+i(x_1^2+x_2^2+x_3^2)^{1/2}\,,\\
\overline{\varphi}(q)&:=\re(q)-i|\im(q)|=x_0-i(x_1^2+x_2^2+x_3^2)^{1/2}
\end{align*}
for every $q=x_0+x_1i+x_2j+x_3k\in\OO_D\setmeno\erre$ with $x_0,x_1,x_2,x_3\in\erre$. Denote by $f_\k|:D\setmeno\erre\funzione X$ and $\partial_\s f|:\OO_D\setminus\erre\funzione X$ the restrictions of $f_\k$ to $D\setminus\erre$ and of $\partial_\s f$ to $\OO_D\setminus\erre$, respectively. By \eqref{equa1}, and by Proposition~\ref{char-Y-holom} and Corollary \ref{C-diff->C-an}$(\mr{i})$ in the Appendix, the function $f_\k:D\funzione X_\k$ is holomorphic and hence real analytic. It follows that the restriction $f_\k|$ of the real analytic function $f_\k$, and the compositions $f_\k|\circ\varphi$ and $f_\k|\circ\overline{\varphi}$ of the real analytic functions $f_\k|$, $\varphi$ and $\overline{\varphi}$ are also real analytic. By \eqref{equa4} and \eqref{equa3}, we have
\[
\textstyle
\partial_\s f|=-(f_\k|\circ\varphi)\frac{\xi\k}{2}+(f_\k|\circ\overline{\varphi})\frac{\xi\k}{2},
\]
so (cf.\ Remark \ref{RA-module}) the restriction $\partial_\s f|$ of $\partial_\s f$ is real analytic as well. By \eqref{equa3} (and Remark~\ref{RA-module}), we also deduce that $F_2:D\to X$ is real analytic, since $f_\k(r-si) = f_{-\k}(r+si)$ for every $r+si\in D$.

Suppose that $\OO_D\cap\erre\neq\vuoto$ and choose arbitrarily a point $r_0\in\OO_D\cap\erre=D\cap\erre$. It remains to show that $\partial_\s f$ is real analytic locally at $r_0$ in $\OO_D$. Up to a translation along the real axis, we can assume $r_0=0$. By \eqref{H} above and Corollary \ref{C-diff->C-an}$(\mr{ii})$ in the Appendix, there exist $\delta_0>0$ and elements $d_{l,k}$ of $X$ for each $l,k\in\enne$ such that $B:=B(0,\delta_0)\subset D$ and the series $\sum_{l, k \in \enne} d_{l,k}r^ls^{2k}$ is absolutely convergent for every $r+si \in B$ and
\begin{equation} \label{eq:F2}
F_2(r+si)s^{-1}=\sum_{l,k \in \enne}d_{l,k}r^l s^{2k} \qquad \forall r+si\in B\setmeno\erre.
\end{equation}

Let $\gamma:\ci\funzione\ci$ be the homeomorphism and let $P$ be the open subset of $\ci$ defined as follows:
\begin{center}
$\gamma(r+si):=r+s^2i\,$ if $\,s\geq0\,$ and $\,\gamma(r+si):=r-s^2i\,$ if $\,s<0$,
\end{center}
and
\begin{center}
$P:=\gamma(B)=\{r+si\in\ci:r^2-\delta^2<s<\delta^2-r^2\}$.
\end{center}
Define the function $G:P\to X$ by setting
\begin{equation}\label{G}
G(r+si):=\sum_{l,k \in \enne}d_{l,k}r^l s^k \qquad \forall r+si\in P,
\end{equation}
where the series $\sum_{l,k \in \enne}d_{l,k}r^l s^k$ converges absolutely for every $r+si\in P$. Observe that $P$ trivially satisfies property \eqref{CCC}. Therefore, we can apply Lemma \ref{vector analytic}(ii), deducing that $G$ is real analytic. By \eqref{eq:F2}, it holds:
\begin{equation}\label{GG}
G(r+s^2i)=F_2(r+si)s^{-1}\qquad \forall r+si\in B\setmeno\erre.
\end{equation}
Since $F_2$ vanishes on $B\cap\erre$, bearing in mind \eqref{GG} and \eqref{equa2}, we have:
\begin{equation}\label{10}
G(q) = \lim_{s \to 0^+} \frac{F_2(q+i\sqrt{s})}{\sqrt{s}} = 
\frac{\partial F_2}{\partial s}(q)=
\frac{\partial F_1}{\partial r}(q)=\frac{\partial f}{\partial r}(q) \qquad \forall q\in B\cap\erre.
\end{equation}

Define the real analytic function $\Gamma:\OO_B\funzione P$ by
\[
\Gamma(q):=\re(q)+i|\im(q)|^2=x_0+i(x_1^2+x_2^2+x_3^2)
\]
for every $q=x_0+x_1i+x_2j+x_3k\in\OO_B$ with $x_0,x_1,x_2,x_3\in\erre$. Observe that $\OO_B$ is an open neighborhood of $r_0$ (an open ball indeed) in $\OO_D$, and the composition $G\circ\Gamma:\OO_B\to X$ of the real analytic functions $G$ and $\Gamma$ is also a real analytic function. By \eqref{equa4}, \eqref{GG} and \eqref{10}, we have:
\[
\partial_\s f(q)=F_2(\re(q)+i|\im(q)|)|\im(q)|^{-1}=G(\re(q)+i|\im(q)|^2)=(G\circ\Gamma)(q)\quad\forall q\in\OO_B\setmeno\erre
\]
and
\[
\partial_\s f(q)=\frac{\partial f}{\partial r}(q)=G(q)=(G\circ\Gamma)(q)\qquad\forall q\in\OO_B\cap\erre.
\]
It follows that $\partial_\s f(q)=(G\circ\Gamma)(q)$ for every $q\in\OO_B$, so $\partial_\s f$ is real analytic on $\OO_B$, as required.
\end{proof}

\begin{Cor}\label{Qq analyt}
Let $D(\A)$ be a right $\quat$-submodule of $X$, let $\A : D(\A) \funzione X$ be a closed right linear operator such that $\rho_\s(\A) \cap \erre \neq \vuoto$ and let us consider the mappings $\S^{{}^{-1}}:\rho_\s(\A)\funzione\Lin^\r(X) : q \longmapsto \Sq(\A)$ and $\Q:\rho_\s(\A)\funzione\Lin^\r(X) : q \longmapsto \Q_q(\A)$. Then, $\S^{{}^{-1}}$ is right slice regular, $\Q=-\partial_\s\S^{{}^{-1}}$ and hence $\Q$ is real analytic.
\end{Cor}

\begin{proof}
First, we recall that $\rho_\s(\A)$ is open in $\quat$ by \cite[Theorem 5.9]{GhiRec18} and, thanks to the condition $\rho_\s(\A) \cap \erre \neq \vuoto$, the function $\S^{{}^{-1}}$ is right slice regular by \cite[Proposition 5.11]{GhiRec18}.

We briefly set $\Delta_q:=\Delta_q(\A)$ and $\Q_q:=\Q_q(\A)$ for every $q \in \rho_\s(\A)$. Since $\Delta_q = \Delta_\qbar$ for every $q \in \quat$, we have that $\Q_q = \Q_\qbar$ for every $q \in \rho_\s(\A)$, so \eqref{resolvent operator} yields
\begin{equation}\label{Q(q+qb)=(Cb-C)}
  \Q_q = (\S^{{}^{-1}}_\qbar - \Sq)(q - \qbar)^{-1}=-\partial_\s\S^{{}^{-1}}(q) \qquad \forall q \in \rho_\s(\A) \setmeno \erre.
\end{equation}
We also briefly set $\G_q:=-\partial_\s\S^{{}^{-1}}(q)=-\partial_\s\Sq$ for every $q\in\rho_\s(\A)$. By Proposition \ref{prop:sd}, the function $\OO_D\funzione X:q\longmapsto \G_q$ is real analytic and hence continuous. Therefore, bearing in mind \eqref{Q(q+qb)=(Cb-C)}, if $r \in \rho_\s(\A) \cap \erre$ and $x \in D(\A^2)$, we have
\begin{align}
  \G_r\Delta_r x 
  & = \lim_{\erre\not\ni q \to r} \G_q\Delta_r x
  =  \lim_{\erre\not\ni q \to r} \Q_q(\Delta_r - \Delta_q + \Delta_q)x \notag \\
  & =  \lim_{\erre\not\ni q \to r} \Q_q(2\re(q-r)\A x + (r^2 - |q|^2)x + \Delta_qx) \notag \\
  & =  \lim_{\erre\not\ni q \to r} \G_q2\re(q-r)\A x + \G_q(r^2 - |q|^2)x + x =  x. \label{GD=I}
\end{align}
In the latter chain of equalities, we used the fact that, thanks to the triangle inequality, the function $\OO_D\funzione X:q \longmapsto \G_q\phi(q)$ is continuous for every given continuous function $\phi : \rho_\s(\A) \funzione \quat$. As $r \in \rho_\s(\A)$, by \eqref{GD=I}, we infer that for every $x \in X$
\begin{equation}\label{DG=I}
\Delta_r\G_r x = \Delta_r \G_r \Delta_r \Q_r x =  \Delta_r \Q_r x=x.
\end{equation}
By \eqref{GD=I} and \eqref{DG=I}, we have $\G_r = \Q_r$ for every $r \in \rho_\s(\A) \cap \erre$. Thus, $\Q=-\partial_\s\S^{{}^{-1}}$ on the whole $\rho_\s(\A)$ and the proof is complete.
\end{proof}

\begin{Prop}\label{P:sviluppoN}
Let $D(\A)$ be a right $\quat$-submodule of $X$ and let $\A : D(\A) \funzione X$ be a closed right linear operator. Choose $N \in \enne$ and $q_0, q \in \rho_\s(\A)$. Recall that $\triangle_{q_0}(q)=q^2 - 2\re(q_0)q + |q_0|^2$ by definition \eqref{triangle_q}. Then, we have:
\begin{align}
  \Sq(\A) 
    & = \sum_{\substack{0\le n \le 2N,\\ n\; \text{even}}} 
           \Q_{q_0}(\A)^{n/2} \S^{{}^{-1}}_{q_0}(\A)\triangle_{q_0}(q)^{n/2} - 
           \!\!\!\sum_{\substack{0\le n \le 2N+1,\\ n\; \text{odd}}} 
               \Q_{q_0}(\A)^{(n+1)/2} (q-q_0)\triangle_{q_0}(q)^{(n-1)/2}\notag\\
    & \sp + \Q_{q_0}(\A)^{N+1}\Sq(\A)  \triangle_{q_0}(q)^{N+1}. \label{sviluppoN}
\end{align} 
\end{Prop}

\begin{proof}
In order to shorten the notation, we set $\Q_q := \Q_q(\A)$ and $\Sq := \Sq(\A)$ for every $q \in \rho_\s(\A)$. We argue by induction on $N\in\enne$. By \eqref{eq:qre}, we have:
\begin{equation}
  \Sq = \S^{{}^{-1}}_{q_0} - \Q_{q_0}(q-q_0)+\Q_{q_0}\Sq\triangle_{q_0}(q),
\end{equation}
which is the claim for $N=0$. Let us now assume that the claim is true for $N\in\enne$, that is, \eqref{sviluppoN}
holds, and let us prove it for $N+1$.
Using again \eqref{eq:qre}, we have:
\begin{align}
  & \Q_{q_0}^{N+1}\Sq  \triangle_{q_0}(q)^{N+1}  \notag \\
  & =  \Q_{q_0}^{N+1}(\S^{{}^{-1}}_{q_0} - \Q_{q_0}(q-q_0)+\Q_{q_0}\Sq\triangle_{q_0}(q))  \triangle_{q_0}(q)^{N+1} \notag \\
  & = \Q_{q_0}^{N+1}\S^{{}^{-1}}_{q_0}\triangle_{q_0}(q)^{N+1} 
        - \Q_{q_0}^{N+2}(q-q_0)\triangle_{q_0}(q)^{N+1}
        + \Q_{q_0}^{N+2}\Sq  \triangle_{q_0}(q)^{N+2} \notag \\
   & = \Q_{q_0}^{2(N+1)/2}\S^{{}^{-1}}_{q_0}\triangle_{q_0}(q)^{2(N+1)/2} 
        - \Q_{q_0}^{(2(N+1) + 1 +1)/2}(q-q_0)\triangle_{q_0}(q)^{(2(N+1)+ 1 -1)/2} \notag \\
   & \phantom{=\ }     + \Q_{q_0}^{N+2}\Sq  \triangle_{q_0}(q)^{N+2}. \label{sviluppoN-aux}
\end{align}
Hence, from the inductive assumption \eqref{sviluppoN} and from \eqref{sviluppoN-aux}, we infer that
\begin{align}
\Sq 
    & = \sum_{\substack{0\le n \le 2(N+1),\\ n\; \text{\it even}}} 
           \Q_{q_0}^{n/2}\S^{{}^{-1}}_{q_0}\triangle_{q_0}(q)^{n/2} - 
           \!\!\!\sum_{\substack{0\le n \le 2(N+1)+1,\\ n\; \text{\it odd}}} 
               \Q_{q_0}^{(n+1)/2}(q-q_0)\triangle_{q_0}(q)^{(n-1)/2} \notag\\
    & \phantom{=\ } + \Q_{q_0}^{N+2}\Sq  \triangle_{q_0}(q)^{N+2}, \notag
\end{align}
which is the claim for $N+1$.
\end{proof}

\begin{Prop}\label{C_q series}
Let $D(\A)$ be a right $\quat$-submodule of $X$, let $\A : D(\A) \funzione X$ be a closed right linear operator, let $q_0\in \rho_\s(\A)$ and let $\mr{U}:=\mr{U}(q_0,\norma{\Q_{q_0}(\A)}{}^{-1/2})$ be the Cassini ball of $\quat$ centered at $q_0$ with radius $\norma{\Q_{q_0}(\A)}{}^{-1/2}$ defined in \eqref{cassini}, that is, $\mr{U}=\{q\in\quat:|\triangle_{q_0}(q)| < \norma{\Q_{q_0}(\A)}{}^{-1}\}$. Then, for every $q \in \mr{U}$, we have:
\begin{equation}\label{sviluppo_inf_temp}
  \Sq(\A) 
   = \sum_{\substack{n \in\enne, \\ n\; \text{even}}} 
           \Q_{q_0}(\A)^{n/2} \S^{{}^{-1}}_{q_0}(\A)\triangle_{q_0}(q)^{n/2} - 
           \sum_{\substack{n \in\enne, \\ n\; \text{odd}}} 
               \Q_{q_0}(\A)^{(n+1)/2} (q-q_0)\triangle_{q_0}(q)^{(n-1)/2}.
\end{equation}
\end{Prop}

\begin{proof}
The assertion follows from Proposition \ref{P:sviluppoN}, from the Weierstrass $M$-test, and from estimate
\[
  \norma{\Q_{q_0}^{N+1}(\A)\Sq(\A)  \triangle_{q_0}(q)^{N+1}}{} 
    \le \norma{\Sq(\A)}{}\left(\frac{|\triangle_{q_0}(q)|}{\norma{\Q_{q_0}(\A)}{}^{-1}}\right)^{N + 1},
\]
letting $N \to \infty$ in \eqref{sviluppoN}.
\end{proof}

\begin{Lem}\label{def F_q}
Let $D(\A)$ be a right $\quat$-submodule of $X$, let
$\A : D(\A) \funzione X$ be a closed right linear operator, let $q_0\in \rho_\s(\A)$ and let $\mr{U}:=\mr{U}(q_0,\norma{\Q_{q_0}(\A)}{}^{-1/2})=\{q\in\quat:|\triangle_{q_0}(q)| < \norma{\Q_{q_0}(\A)}{}^{-1}\}$. Define the operators $\C_n\in\Lin^\r(X)$ and the functions $\qspow:\quat\lra\quat$ by
\begin{equation}\label{Bn}
  \C_n
  :=
  \begin{cases}
    \Q_{q_0}(\A)^k & \text{if $n=2k$, \  $k \in \enne$,} \\
    \Q_{q_0}(\A)^k\S^{{}^{-1}}_{q_0}(\A) & \text{if $n=2k+1$, \  $k \in \enne$,}
  \end{cases}
\end{equation}
and
\begin{equation}\label{pn}
  \qspow(q) :=
  \begin{cases}
    \triangle_{q_0}(q)^k & \text{if $n=2k$, \ $k \in \enne$,} \\
    (q-q_0)\triangle_{q_0}(q)^k & \text{if $n=2k+1$, \ $k \in \enne$.}
  \end{cases}
\end{equation}
 
Then, for every $q \in \mr{U}$, the series
\begin{equation}\label{tot conv series}
  \F_q := \sum_{n=0}^{\infty}(-1)^n \C_{n+1}
  \qspow(q) 
\end{equation}
is absolutely convergent in $\Lin^\r(X)$, and the mapping 
$\F : \mr{U}\funzione \Lin^\r(X) : q \longmapsto \F_q$ defines a right slice regular function, according to Definition \ref{D:Xslice regularity}. Moreover, for every $q \in \mr{U}$, we have:
$\F_q x \in D(\A^2)\subset D(\A)$ for every $x \in X$, $\A\F_q \in \Lin^r(X)$, $\A^2\F_q \in \Lin^r(X)$, and 
\begin{equation}\label{AF,A2F}
  \A\F_q = \sum_{n=1}^\infty (-1)^n\A\C_{n+1} \qspow(q), \qquad
  \A^2 \F_q = \sum_{n=0}^\infty (-1)^n\A^2\C_{n+1} \qspow(q),
\end{equation}
with absolutely convergent series in $\Lin^\r(X)$.
\end{Lem}

\begin{proof}
We briefly set $\Q_q := \Q_q(\A)$ and $\Sq := \Sq(\A)$ for every $q \in \rho_\s(\A)$. If $n = 2k$ for some $k \in \enne$, we have
\begin{align}
  \norma{\C_{n+1} \qspow(q)}{} 
  & =   \norma{\Q_{q_0}^k \S^{{}^{-1}}_{q_0}\triangle_{q_0}(q)^k}{}  \le (\norma{\Q_{q_0}}{}|\triangle_{q_0}(q)|)^k \norma{\S^{{}^{-1}}_{q_0}}{}. \notag
\end{align}
If $n = 2k+1$ for some $k \in \enne$, we have
\begin{align}
  \norma{\C_{n+1} \qspow(q)}{} 
  & =   \norma{\Q_{q_0}^{k+1} (q-q_0)\triangle_{q_0}(q)^k}{}  \le (\norma{\Q_{q_0}}{}|\triangle_{q_0}(q)|)^k\norma{\Q_{q_0}}{}|q-q_0|. \notag
\end{align}
Thus, the series \eqref{tot conv series} is absolutely convergent if $q\in\mr{U}$. Moreover, the functions $\qspow$ are polynomials with coefficients on the left, hence they are right slice regular (cf.\ \cite[Section~5.1]{GhiRec18}). It immediately follows that $\F$ is also right slice regular.

Pick $q \in\mr{U}$ and $x \in X$. Since $(-1)^n\C_{n+1}\qspow(q)x \in D(\A^2)$ for every $n \in \enne$, it holds:
\begin{equation}
  x_N := \sum_{n=0}^N (-1)^n\C_{n+1} \qspow(q)x \in D(\A^2) \subset D(\A) \qquad \forall N \in \enne.
\end{equation}
Observe that the series $\sum_{n=0}^\infty(-1)^n\A\C_{n+1} \qspow(q)x$ converges in $X$. Indeed, for $n = 2k$, $k \in \enne$, $k\geq1$, bearing in mind that $\A\Q_{q_0} \in \Lin^r(X)$ by Lemma \ref{AQ Lin}, we have
\begin{align}\label{stima 1}
  \norma{\A\C_{n+1} \qspow(q)x}{} 
  & =   \norma{\A\Q_{q_0}^k \S^{{}^{-1}}_{q_0}\triangle_{q_0}(q)^kx}{} \notag \\
  & \le  \norma{\A\Q_{q_0}}{}(\norma{\Q_{q_0}}{}|\triangle_{q_0}(q)|)^{k-1} \norma{\S^{{}^{-1}}_{q_0}}{}|\triangle_{q_0}(q)|\norma{x}{}, 
\end{align}
and, for $n = 2k+1$, $k\in \enne$, $k\geq1$, we have
\begin{align}\label{stima 2}
  \norma{\A\C_{n+1} \qspow(q)x}{} 
  & =   \norma{\A\Q_{q_0}^{k+1} (q-q_0)\triangle_{q_0}(q)^kx}{} \notag \\
  & \le \norma{\A\Q_{q_0}}{}(\norma{\Q_{q_0}}{}|\triangle_{q_0}(q)|)^{k-1}
   \norma{\Q_{q_0}}{}|q-q_0||\triangle_{q_0}(q)|\norma{x}{}. 
\end{align}
Therefore,
\[
  \lim_{N\to \infty} \A x_N = \lim_{N\to \infty}\sum_{n=0}^N (-1)^n\A\C_{n+1} \qspow(q)x = 
  \sum_{n=0}^\infty (-1)^n\A\C_{n+1} \qspow(q)x,
\]
and, by the closedness of $\A$, we infer that 
\[
  \F_q x = \lim_{N\to \infty} x_N = \sum_{n=0}^\infty (-1)^n\C_{n+1} \qspow(q)x  \in D(\A),
\] 
and
\begin{equation}\label{AF=}
  \A\F_q x = \sum_{n=1}^\infty (-1)^n\A\C_{n+1} \qspow(q)x.
\end{equation}

We also have that 
\begin{equation}
  \A x_N = \sum_{n=0}^N (-1)^n\A\C_{n+1} \qspow(q)x \in D(\A) 
  \qquad \forall N \in \enne
\end{equation}
and the series $\sum_{n=0}^\infty(-1)^n\A^2\C_{n+1} \qspow(q)x$ converges in $X$. Indeed, it holds 
$\A^2\Q_{q_0}^2 = \A\A\Q_{q_0}\Q_{q_0} = \A\Q_{q_0}\A\Q_{q_0} \in \Lin^r(X)$, thus for 
$n = 2k$, $k \in \enne$, $k\geq2$, we have
\begin{align}\label{stima 3}
  \norma{\A^2\C_{n+1} \qspow(q)x}{} 
  & =   \norma{\A^2\Q_{q_0}^k \S^{{}^{-1}}_{q_0}\triangle_{q_0}(q)^kx}{} \notag \\
  & \le   \norma{\A^2\Q_{q_0}^2}{}(\norma{\Q_{q_0}}{}|\triangle_{q_0}(q)|)^{k-2} \norma{\S^{{}^{-1}}_{q_0}}{}|\triangle_{q_0}(q)|^{2}\norma{x}{}.
\end{align}
and for $n = 2k+1$, $k \in \enne$, $k\geq2$, we have
\begin{align}\label{stima 4}
  \norma{\A^2\C_{n+1} \qspow(q)x}{} 
  & =   \norma{\A^2\Q_{q_0}^{k+1} (q-q_0)\triangle_{q_0}(q)^kx}{} \notag \\
  & \le \norma{\A^2\Q_{q_0}^2}{}(\norma{\Q_{q_0}}{}|\triangle_{q_0}(q)|)^{k-2}
   \norma{\Q_{q_0}}{}|q-q_0||\triangle_{q_0}(q)|^2\norma{x}{}.
\end{align}
Therefore,
\[
  \lim_{N \to \infty} \A^2 x_N = \lim_{N \to \infty}\sum_{n=0}^N (-1)^n\A^2\C_{n+1} \qspow(q)x = 
   \sum_{n=0}^\infty (-1)^n\A^2\C_{n+1} \qspow(q)x,
\]
and, by the closedness of $\A$, we get
\[
  \A\F_qx = \sum_{n=0}^\infty (-1)^n\A\C_{n+1} \qspow(q)x = \lim_{N \to \infty} \A x_N \in D(\A),
\]
and also
\begin{equation}\label{A2F=}
  \A^2 \F_qx = \sum_{n=0}^\infty (-1)^n\A^2\C_{n+1} \qspow(q)x.
\end{equation}
Finally, estimates \eqref{stima 1}, \eqref{stima 2}, \eqref{stima 3} and \eqref{stima 4} also prove that the series in \eqref{AF,A2F} are absolutely convergent series in $\Lin^\r(X)$. Hence, \eqref{AF,A2F} holds true by virtue of \eqref{AF=} and \eqref{A2F=}.
\end{proof}

\begin{Lem}\label{GinDA2}
Let $D(\A)$, $\A:D(\A)\funzione X$, $q_0\in\rho_\s(\A)$, $\mr{U}$ and $\F : \mr{U}\funzione \Lin^\r(X):q\longmapsto \F_q$ be as in Lemma \ref{def F_q}. Define 
$\G : \mr{U}\funzione X:q\longmapsto \G_q$ by $\G_q:=-\partial_\s\F(q)$. Then, $\G_qx \in D(\A^2)$ for every $q\in\mr{U}$ and $x \in X$.
\end{Lem}

\begin{proof}
We first observe that, for $h=1,2$, the mapping 
$\A^h\F : \mr{U} \funzione \Lin^\r(X): q \longmapsto \A^h\F_q$
is right slice regular, because the polynomial functions $\qspow:\quat\funzione\quat$ are right slice regular (cf.\ \cite[Section~5.1]{GhiRec18}). By Proposition \ref{prop:sd}, we can define the real analytic functions $\M : \mr{U} \funzione \Lin^\r(X): q \longmapsto \M_q$ and $\N : \mr{U} \funzione \Lin^\r(X): q \longmapsto \N_q$ by $\M:=-\partial_\s(\A\F)$ and $\N:=-\partial_\s(\A^2\F)$. Since a real analytic function is also continuous, we have:
\begin{equation}
  \M_q := 
  \begin{cases}
    (\A\F_\qbar - \A\F_q)(q - \qbar)^{-1} 
      & \text{if $q \in \mr{U} \setmeno \erre$}, \\
    \lim_{\erre\not\ni p \to q} (\A\F_\pbar - \A\F_p)(p - \pbar)^{-1}
      & \text{if $q \in \mr{U} \cap \erre$},
  \end{cases}
\end{equation}  
and
\begin{equation}
  \N_q := 
  \begin{cases}
    (\A^2\F_\qbar - \A^2\F_q)(q - \qbar)^{-1} 
      & \text{if $q \in \mr{U} \setmeno \erre$}, \\
    \lim_{\erre\not\ni p \to q} (\A^2\F_\pbar - \A^2\F_p)(p - \pbar)^{-1}
      & \text{if $q \in \mr{U} \cap \erre$},
  \end{cases}
\end{equation}  
Let $x\in X$. Clearly, $\G_q x=(\F_\qbar - \F_q)(q - \qbar)^{-1}x  \in D(\A^2)$ for every $q \in \mr{U} \setmeno \erre$ by Lemma \ref{def F_q}. We are left to deal with the case $q=r \in \mr{U} \cap \erre$. Let us apply the closedness of $\A$ along a sequence 
$p=p_n \in \mr{U} \setmeno \erre$ converging to $r$. It holds: $(\F_\pbar - \F_p)(p-\pbar)^{-1}x \in D(\A^2)$ for every $p$, $\lim_{p\to r}(\F_\pbar - \F_p)(p-\pbar)^{-1}x = \G_rx$, and
$\lim_{p\to r} \A(\F_\pbar - \F_p)(p-\pbar)^{-1}x = \lim_{p\to r} (\A\F_\pbar - \A\F_p)(p-\pbar)^{-1}x = \M_rx$. Therefore, by the closedness of $\A$, we infer that $\G_rx \in D(\A)$ and $\A\G_rx = \M_rx$. Moreover, we have: $(\A\F_\pbar - \A\F_p)(p-\pbar)^{-1}x \in D(\A)$ for every $p\in\mr{U} \setmeno \erre$, $\lim_{p\to r}(\A\F_\pbar - \A\F_p)(p-\pbar)^{-1}x = \M_rx$, and $\lim_{p\to r}\A(\A\F_\pbar - \A\F_p)(p-\pbar)^{-1}x =\lim_{p\to r}(\A^2\F_\pbar - \A^2\F_p)(p-\pbar)^{-1}x = \N_rx$, therefore the closedness of $\A$ yields
$\M_rx \in D(\A)$, so $\A\G_rx = \M_rx\in D(\A)$. It follows that $\G_r x \in D(\A^2)$.
\end{proof}

We are now ready to prove our second main result.

\begin{proof}[Proof of Theorem \ref{T:main}]
Let $D(\A)$, $\A : D(\A) \funzione X$ with $\rho_\s(\A) \cap \erre \neq \vuoto$, $q_0\in \rho_\s(\A)$ and $R:=\norma{\Q_{q_0}(\A)}{}^{-1/2}>0$ be as in the statement of the theorem. Set $\mr{U}:=\mr{U}(q_0,R)$. As we have already observed in the proof Corollary \ref{Qq analyt}, $\rho_\s(\A)$ is open by \cite[Theorem 5.9]{GhiRec18} and, thanks to the condition $\rho_\s(\A) \cap \erre \neq \vuoto$, the function $\rho_\s(\A)\funzione X:q \longmapsto \Sq$ is right slice regular by \cite[Proposition 5.11]{GhiRec18}. Let $\F : \mr{U} \funzione \Lin^\r(X) : q \longmapsto F_q$ be the right slice regular function defined as in \eqref{tot conv series}, that is,
\begin{equation}\label{Fq}
  \F_q :=  \sum_{n=0}^{\infty}(-1)^n \C_{n+1}\qspow(q)
  \qquad \forall q \in \mr{U},
\end{equation}
where $\B_{n}$ and $\qspow(q)$ are defined in \eqref{Bn} and \eqref{pn}. By Proposition \ref{C_q series}, we infer that
\begin{equation}\label{F=C}
  \F_q = \Sq \qquad \forall q \in \rho_\s(\A) \cap \mr{U}.
\end{equation}
Thanks to Proposition \ref{prop:sd}, we can define the real analytic function $\G : \mr{U} \funzione \Lin^\r(X) : q \longmapsto \G_q$ by  $\G_q:=-\partial_\s\F(q)$. Therefore, by Corollary \ref{Qq analyt}, we have: 
\begin{equation}\label{G=Q}
  \G_q = \Q_q \qquad \forall q \in \rho_\s(\A) \cap \mr{U},
\end{equation}
where $\Q_q:=\Q_q(\A)$. Thanks to Lemma \ref{GinDA2}, we know that $\G_q(X) \subseteq D(\A^2)$ for every 
$q \in \mr{U}$. Thus, from \eqref{G=Q}, we deduce:
\begin{equation}\label{analyt on the right}
  \Q_{q_0}(\Delta_q \G_q - \Id) = 0 \qquad 
  \forall q \in \rho_\s(\A) \cap \mr{U}
\end{equation}
and
\begin{equation}\label{analyt on the left}
  (\G_q \Delta_q - \Id)\Q_{q_0} = 0 \qquad \forall q \in \rho_\s(\A) \cap \mr{U}.
\end{equation}
Observe that $\rho_\s(\A) \cap \mr{U}$
is an open subset of $\mr{U}$ containing $q_0$, so it is non-empty.
For every $q \in \mr{U}$, it holds: 
\begin{align}
  \Q_{q_0}\Delta_q \G_q  
  & = \Q_{q_0}(\Delta_{q} - \Delta_{q_0} + \Delta_{q_0})  \G_q  \notag \\
  & =  \Q_{q_0}\big(2\re(q_0-q)\A + (|q|^2 - |q_0|^2)\Id + \Delta_{q_0}\big)\G_q   \notag \\
  & =  (2\re(q_0-q)\Q_{q_0}\A  + (|q|^2 - |q_0|^2)\Q_{q_0} + \Q_{q_0}\Delta_{q_0})\G_q   \notag \\
  & =  (2\re(q_0-q)\A\Q_{q_0}  + (|q|^2 - |q_0|^2)\Q_{q_0}  + \Id)\G_q x  \notag \\
  & =  (2\re(q_0-q)(\Q_{q_0}\qbar_0-\S^{{}^{-1}}_{q_0})  + (|q|^2 - |q_0|^2)\Q_{q_0} + \Id)\G_q  \notag \\
  & =  2(\Q_{q_0}\qbar_0-\S^{{}^{-1}}_{q_0})\G_q\re(q_0-q) + \Q_{q_0}\G_q(|q|^2 - |q_0|^2)    
         + \G_q,  \notag 
\end{align}
so the function $\mr{U}\funzione X:q\longmapsto \Q_{q_0}(\Delta_q \G_q - \Id)=\Q_{q_0}\Delta_q \G_q-\Q_{q_0}$ is real analytic, whose vanishing set contains the non-empty open subset $\rho_\s(\A)\cap\mr{U}$ of $\mr{U}$ by \eqref{analyt on the right}. Lemma \ref{vector analytic} now implies that
\[
  \Q_{q_0}(\Delta_q \G_q - \Id) = 0 \qquad \forall q \in \mr{U},
\]
which proves that $\G_q$ is a right inverse of $\Delta_q$ for every $q \in\mr{U}$. Similarly, we have:
\begin{align}
  \G_q \Delta_q\Q_{q_0} 
  & =  \G_q(\Delta_{q} - \Delta_{q_0} + \Delta_{q_0}) \Q_{q_0}  \notag \\
  & =  \G_q\big(2\re(q_0-q)\A + (|q|^2 - |q_0|^2)\Id + \Delta_{q_0}\big)\Q_{q_0}   \notag \\
  & =  \G_q(2\re(q_0-q)\A\Q_{q_0}  + (|q|^2 - |q_0|^2)\Q_{q_0} + \Id   \notag \\
  & =  \G_q(\Q_{q_0}\qbar_0-\S^{{}^{-1}}_{q_0})2\re(q_0-q)  + \G_q\Q_{q_0}(|q|^2 - |q_0|^2) + \G_q  \notag  
\end{align}
so the function $\mr{U}\funzione X:q\longmapsto (\G_q \Delta_q - \Id)\Q_{q_0}= \G_q\Delta_q\Q_{q_0}-\Q_{q_0}$ is real analytic, whose vanishing set contains the non-empty open subset $\rho_\s(\A)\cap\mr{U}$ of $\mr{U}$ by \eqref{analyt on the left}. Lemma \ref{vector analytic} implies that
\[
  (\G_q\Delta_q - \Id)\Q_{q_0} = 0 \qquad \forall q \in \mr{U},
\]
which proves that $\G_q$ is a left inverse of $\Delta_q$ for every $q \in \mr{U}$. It follows that $\mr{U} \subset \rho_\s(\A)$ and the theorem follows from \eqref{Fq} and \eqref{F=C}.
\end{proof}

Let us give the proofs of the corollaries.

\begin{proof}[Proof of Corollary \ref{cor3}]
Immediate from Theorem \ref{T:main}$(\mr{ii})$ and Corollary \ref{Qq analyt}.
\end{proof}

\begin{proof}[Proof of Corollary \ref{cor1}]
Immediate from Theorem \ref{T:main}$(\mr{i})$.
\end{proof}

\begin{proof}[Proof of Corollary \ref{cor2}]
If $q_0\in\rho_\s(\A)=\quat\setmeno\sigma_\s(\A)$, then the continuity of the mapping $\rho_\s(\A)\funzione\Lin^r(X):q\longmapsto \Q_q(\A)$ proved in Corollary \ref{cor3} ensures that $\lim_{n\to \infty}\norma{\Q_{p_n}(\A)}{} = \norma{\Q_{q_0}(\A)}{}<\infty$.

Suppose that $\lim_{n\to \infty}\norma{\Q_{p_n}(\A)}{}\neq\infty$ and $q_0\in\sigma_\s(\A)$. Then, there exist $L\in\erre$ and a subsequence $\{p_{n_k}\}_{k\in\enne}$ of $\{p_n\}_{n\in\enne}$ such that $\lim_{k\to \infty}\norma{\Q_{p_{n_k}}(\A)}{}=L\geq0$. By Corollary \ref{cor1}, we know that $\mr{u}(\sigma_\s(\A),p_{n_k})\geq\norma{\Q_{p_{n_k}}(\A)}{}^{-1/2}$ so $\mr{u}(\sigma_\s(\A),p_{n_k})\geq(L+1)^{-1/2}>0$ for every $k\geq h$, where $h$ is a sufficiently large natural number. As $q_0\in\sigma_\s(\A)$, we have that $\mr{u}(\sigma_\s(\A),p_{n_k})\to\mr{u}(\sigma_\s(\A),q_0)=0$ as $k\to\infty$. Thus, we deduce that $0\geq(L+1)^{-1/2}$, which is a contradiction.
\end{proof}

\section{Appendix: Vector-valued analytic functions}\label{appendix}

The aim of this appendix is to prove that a vector-valued $\ci$-differentiable function is also real analytic in the sense specified by Definition \ref{real analytic}. We start with some characterizations of vector-valued holomorphy which are not often explicitly stated and proved in the literature.

\begin{Prop}\label{char-Y-holom}
Let $D$ be a non-empty open subset of $\ci$, let $Y$ be a complex Banach space and let $g : D \funzione Y$ be a function. Then, the following conditions are equivalent:
\begin{itemize}
\item[$(\mr{i})$] $g$ is holomorphic, that is , $g$ is $\ci$-differentiable at each point of $D$.
\item[$(\mr{ii})$] $g$ is of class $C^1$ (in the real sense) and $\frac{\partial g}{\partial \overline{z}}:=\frac{1}{2}\big(\frac{\partial g}{\partial r} + i\frac{\partial g}{\partial s}\big) = 0$ on $D$, where $z=r+si$  with $r,s\in\erre$.
\item[$(\mr{iii})$] $g$ is analytic on $D$. More precisely, if $z_0 \in D$ and $B(z_0,\delta_0):=\{z \in \ci:|z-z_0| < \delta\}$ is contained in $D$ for some $\delta_0 > 0$, then $g(z) = \sum_{n = 0}^\infty (z-z_0)^n a_n$ for every $z \in B(z_0,\delta_0)$, where $a_n = (2\pi i)^{-1}\int_{\partial B(z_0,\delta)} (z-z_0)^{-n-1}g(z)\de z$ for any $\delta \in \opint{0,\delta_0}$ (the integral being meant in the Bochner sense, cf.\ e.g.\ \cite[Chapter~VI]{Lan93}). 
\end{itemize}
\end{Prop}

\begin{proof}
Assume that $(\mr{i})$ holds. Then, for every linear continuous functional $\Phi : Y \funzione \ci$, we have that $\Phi \circ g : D \funzione \ci$ is holomorphic in the usual sense. By the Cauchy formula, we deduce $\Phi(g(z)) = (2\pi i)^{-1}\int_\gamma (w-z)^{-1}\Phi(g(w))\de w=\Phi\big((2\pi i)^{-1}\int_\gamma (w-z)^{-1}g(w)\de w\big)$ for every simple arc $\gamma$ in $D$ whose image is a circle containing $z$ in its interior (the second integral is meant in the Bochner sense). The complex Hahn-Banach theorem now implies the validity of the following vector Cauchy formula: 
$g(z) = (2\pi i)^{-1}\int_\gamma (w-z)^{-1}g(w)\de w$. Arguing exactly as in the scalar case, this formula allows to prove statement $(\mr{iii})$ (observe that the the integral defining $a_n$ is independent of $r$ by virtue of the holomorphy of $\Phi \circ g$ and by the complex Hahn-Banach theorem). Vice versa power series with coefficients in $Y$ are $\ci$-differentiable (as the same direct argument used in the scalar case shows), therefore $(\mr{iii})$ implies $(\mr{i})$.

Assuming again $(\mr{i})$ and arguing as in the scalar case, we find that $g' = \frac{\partial g}{\partial r} = - i\frac{\partial g}{\partial s}$ so $\frac{\partial g}{\partial \overline{z}}=0$ on $D$. This proves $(\mr{ii})$. A standard computation gives implication $(\mr{ii})\Longrightarrow(\mr{i})$.
\end{proof}

\begin{Cor}\label{C-diff->C-an}
Let $D$ be a non-empty open subset of $\ci$, let $(Y,\norma{\cdot}{})$ be a complex Banach space and let $g : D \funzione Y$ be a holomorphic function. We have:
\begin{itemize}
\item[$(\mr{i})$] $g$ is real analytic. More precisely, $g$ has the following property: Pick any $z_0=r_0 + s_0i \in D$ and $\delta_0>0$ such that $B(z_0,2\delta_0)\subset D$. Then, there are elements $c_{l,m}$ of $Y$ for each $l,m\in\enne$ such that
\begin{equation}\label{real power series}
  g(z) = \sum_{l, m \in \enne} (r-r_0)^l(s-s_0)^mc_{l,m}  \qquad 
  \forall z=r+si \in B(z_0,\delta_0),\  r, s \in \erre,
\end{equation} 
where the convergence in \eqref{real power series} is meant in the sense of summability, that is, for every $\eps > 0$, there exists a finite set $F_\eps \subset \enne^2$ such that 
$\norma{g(z) - \sum_{(l,m) \in F} (r-r_0)^l(s-s_0)^mc_{l,m}}{} < \eps$ for every finite set $F\subset\enne^2$ with $F_\eps \subset F$. Moreover, for every $r+si\in B(z_0,\delta_0)$, the series $\sum_{l, m \in \enne} (r-r_0)^l(s-s_0)^mc_{l,m}$ converges absolutely and, for every bijective mapping $(\alpha,\beta):\enne\to\enne^2$, the rearranged series $\sum_{n\in\enne} (r-r_0)^{\alpha(n)}(s-s_0)^{\beta(n)}c_{\alpha(n),\beta(n)}$ converges absolutely. 
\item[$(\mr{ii})$] Suppose that $D$ is invariant under the complex conjugation and $0\in D$. Define the function $h:D\setmeno\erre\funzione Y$ by $h(z):=(z-\overline{z})^{-1}(g(z)-g(\overline{z}))$. Pick $\delta_0>0$ such that $B(0,2\delta_0)\subset D$. Then, there are elements $d_{l,k}$ of $Y$ for each $l,k\in\enne$ such that the series $\sum_{l, k \in \enne} r^ls^{2k}d_{l,k}$ is absolutely convergent for every $r+si \in B(0,\delta_0)$ with $r,s\in\erre$, and its sum is equal to $h(r+si)$ if $r+si \in B(0,\delta_0)\setmeno\erre$.
\end{itemize}
\end{Cor}

\begin{proof}
By equivalence $(\mr{i})\Longleftrightarrow(\mr{iii})$ in Proposition \ref{char-Y-holom}, we know that $g(z) = \sum_{n = 0}^\infty (z-z_0)^na_n $ for every $z \in B(z_0,2\delta_0)$, where $a_n = (2\pi i)^{-1}\int_{\partial B(z_0,\eta)} (z-z_0)^{-n-1}g(z)\de z$ for any $\eta \in \opint{0,2\delta_0}$. By the binomial formula, we find 
\begin{equation}\label{g(z)=series}
  g(z) = \sum_{n = 0}^\infty  \sum_{m=0}^n \binom{n}{m}i^m(r-r_0)^{n-m}(s-s_0)^ma_n
  \qquad \forall z=r+si \in B(z_0,2\delta_0).
\end{equation}
Take $\delta_1,\delta\in\erre$ such that $0 < \delta_1 < \delta < \delta_0$ and set $M:=\max_{\zeta\in \partial B(z_0,2\delta)}\norma{g(\zeta)}{}$. Since $a_n = (2\pi i)^{-1}\int_{\partial B(z_0,2\delta)} (z-z_0)^{-n-1}g(z)\de z$, it follows that $\norma{a_n}{} \le M(2\delta)^{-n}$ for each $n \in \enne$,
and
\begin{align}
 &\sum_{m=0}^n \binom{n}{m}|r-r_0|^{n-m}|s-s_0|^m \norma{a_n}{}= (|r-r_0|+ |s-s_0|)^n\norma{a_n}{}  \notag \\
   & \le \Big(\frac{\delta_1+\delta_1}{2\delta}\Big)^nM = \Big(\frac{\delta_1}{\delta}\Big)^nM  \qquad \forall r+si \in B(z_0,\delta_1).\notag
\end{align}
Therefore, if we set $c_{l,m}:=\binom{l+m}{m}i^m a_{l+m}$, the series $\sum_{l,m \in \enne}(r-r_0)^l(s-s_0)^mc_{l,m}$ is summable and its sum is $g(z)$ for every $z=r+si\in B(z_0,\delta_0)$. Actually, each of its rearrangements $\sum_{n\in\enne}(r-r_0)^{\alpha(n)}(s-s_0)^{\beta(n)}c_{\alpha(n),\beta(n)}$ is absolutely convergent to $g(z)$ for every $z=r+si\in B(z_0,\delta_0)$. This proves $(\mr{i})$.

Let us show $(\mr{ii})$. Suppose now that $D$ is invariant under the complex conjugation, and $z_0=0\in D$. By \eqref{g(z)=series}, we have:
\begin{align}
h(z)=(z-\overline{z})^{-1}(g(z)-g(\overline{z}))&=(2si)^{-1}\sum_{n=0}^\infty\sum_{m=0}^n\binom{n}{m}i^mr^{n-m}s^m(1-(-1)^m)a_n \notag \\
&=\sum_{n=1}^\infty\sum_{k\in\enne,\ 2k+1\leq n}\binom{n}{2k+1}(-1)^kr^{n-2k-1}s^{2k}a_n 
\label{h=}
\end{align}
for every $z=r+si\in B(0,2\delta_0)\setmeno\erre$. Consider again $\delta_1,\delta\in\erre$ such that $0 < \delta_1 < \delta < \delta_0$ and $M=\max_{\zeta\in \partial B(z_0,2\delta)}\norma{g(\zeta)}{}$. For every $r+si\in B(0,\delta_1)$, we have:
\begin{align*}
&\sum_{n=1}^\infty\sum_{k\in\enne,\ 2k+1\leq n}\binom{n}{2k+1}|r|^{n-2k-1}|s|^{2k}\norma{a_n}{}\leq\sum_{n=1}^\infty\sum_{k\in\enne,\ 2k+1\leq n}\binom{n}{2k+1}\delta_1^{n-1}\frac{M}{(2\delta)^n}\\
&\leq\frac{M}{\delta_1}\sum_{n=1}^\infty\left(\frac{\delta_1}{\delta}\right)^n\frac{1}{2^n}\sum_{k\in\enne,\ 2k+1\leq n}\binom{n}{2k+1}\leq\frac{M}{\delta_1}\sum_{n=0}^\infty\left(\frac{\delta_1}{\delta}\right)^n\frac{1}{2^n}\sum_{h=0}^n\binom{n}{h}=\frac{M}{\delta_1}\sum_{n=0}^\infty\left(\frac{\delta_1}{\delta}\right)^n.
\end{align*}
Thus, if we set $d_{l,k}:=\binom{l+2k+1}{2k+1}(-1)^ka_{l+2k+1}$ for every $l,k\in\enne$, then the series $\sum_{l, k \in \enne} r^ls^{2k}d_{l,k}$ is absolutely convergent for every $r+si \in B(0,\delta_0)$, and, 
thanks to \eqref{h=}, its sum is equal to $h(r+si)$ if $r+si \in B(0,\delta_0)\setmeno\erre$, as required.
\end{proof}



\end{document}